\documentclass[QTF8,fontset=windows,final, leqno]{siamltex}
\usepackage{booktabs} 
\usepackage[noadjust]{cite}
\usepackage{ulem}
\usepackage{ctex}
\usepackage{amsmath}
\allowdisplaybreaks[4]
\usepackage{cancel}
\usepackage{multirow}
\usepackage{graphicx}
\usepackage{amssymb}
\usepackage{color}
\usepackage{tikz}
\usepackage{hyperref}
\usepackage{float}
\usepackage{caption}
\usepackage[toc,page]{appendix}
\usepackage{mathrsfs}
\usepackage{latexsym, bm}
\usepackage{bm}
\numberwithin{equation}{section}
\newtheorem{remark}{Remark}[section]

\newcommand{\vertiii}[1]{{\left\vert\kern-0.23ex\left\vert\kern-0.23ex\left\vert #1
		\right\vert\kern-0.23ex\right\vert\kern-0.23ex\right\vert}}
\renewcommand\abstractname{Abstract}

\renewcommand\figurename{Figure}
\renewcommand\tablename{Table}

\normalsize

\renewcommand\refname{References}

\def\Norm#1#2{\left\|\,#1\,\right\|_{#2}}

\def\no{{\nonumber}}

\newcommand{\e}[1]{\mathbf{e}^{#1}}

\newcommand{\mN}{\mathcal{N}}
\newcommand{\timesr}{\times_r}
\newcommand{\mL}{\mathcal{L}}

\newcommand{\il}[2]{\left\langle #1,#2\right\rangle_2}
\newcommand{\norm}[1]{\left\|#1\right\|}

\newcommand{\Lh}{\mathcal{L}_h}
\newcommand{\Nh}{\mathcal{N}_h}
\newcommand{\PN}{P_N}
\newcommand{\VN}{V_N}
\newcommand{\N}{\mathcal{N}}

\newcommand{\dt}{\tau} 
\newcommand{\inner}[2]{\left( #1, #2 \right)}
\newcommand{\fnorm}[1]{\|#1\|_F}
\newcommand{\innerx}[2]{\left( #1, #2 \right)_{\mathcal{X}}}
\newcommand{\gradcurl}[1]{\nabla \times #1}

\def\ba{\mathbf{a}}

\newcommand{\tensor}[1]{\bm{\mathrm{#1}}}
\newcommand{\vect}[1]{\bm{#1}}
\newcommand{\vu}[1]{\hat{\bm{#1}}}
\def\mL{\mathcal{L}}
\def\mN{\mathcal{N}}

\def\ba{\begin{align}}
\def\da{\end{align}}

\CTEXoptions[today=old]
\begin{document}


\renewcommand{\thefootnote}{\fnsymbol{footnote}}
\title{On the maximum bound principle and energy dissipation of exponential time differencing methods for the chiral liquid crystal blue phases\footnote{Last update: \today}}

\author{Wenshuai Hu
	\and
	Guanghua Ji\thanks{Laboratory of Mathematics and Complex Systems, Ministry of Education and School of Mathematical Sciences,
		Beijing Normal University, Beijing 100875, China.\newline
				\texttt{202431130052@mail.bnu.edu.cn}, \texttt{ghji@bnu.edu.cn, Corresponding author}}
}

\maketitle
\begin{abstract}
    The blue phases  are fascinating and complex states of chiral liquid crystals which can be modeled by a comprehensive framework of the Landau-de Gennes theory.
In this paper, we develop and analyze a first order and a second order exponential time differencing schemes for the gradient flow of the  chiral liquid crystal blue phases, which preserve the maximum bound principle  and energy dissipation unconditionally.
The fully discrete schemes are obtained coupled with the Fourier spectral method in space.
And we propose a novel matrix-form Helmholtz basis transformation method  to diagonalize the combined operator of the Laplacian and the curl operator, which is the key step in the implementation of the proposed schemes.
Then by constructing  auxiliary functions, we drive the $L^\infty$ boundedness of the numerical solutions, the energy dissipation and the error estimates in $L^2$ and $L^\infty$ norm.
Various numerical experiments are presented to validate the theoretical results and demonstrate the effectiveness of the proposed methods in simulating the dynamics of blue phases in chiral liquid crystals.
\end{abstract}


\pagestyle{myheadings}
\thispagestyle{plain}
\markboth{Wenshuai Hu
	AND
	Guanghua Ji}{Exponential time differencing methods for the blue phases of chiral liquid crystals}

\section{Introduction}~\\
The blue phases(BPs) of chiral liquid crystals are a fascinating class of soft materials that exhibit unique optical and mechanical properties \cite{BluePhasesLiquid1985,PhysRevAblue1987}. They are characterized by a three-dimensional cubic lattice structure formed by the self-assembly of chiral molecules, resulting in a periodic arrangement of defects known as disclination lines \cite{goodby1991chirality,salamonczyk2019multi}.
To accurately describe the phase transitions and topological defect evolution in such complex structures, the Landau-de Gennes theory provides a powerful theoretical framework by introducing a second order tensor $\mathrm{Q}$, which can also be extended to the chiral blue-phases \cite{deGennes1974,PhysRevX_elastic2017,Ravnik01102009,bahr2001chirality}. This paper is based on these basis theories and conducts research starting from the classical blue-phase model \cite{bahr2001chirality}.
Following \cite{deGennes1974,bahr2001chirality},  the Landau-de Gennes free energy  functional of BPs can be expressed as
\begin{alignat}{2}
	F[Q]=\int_\Omega \left[ \frac{L_1}{2}\|\nabla Q\|_F^2+\frac{L_4}{2}Q:\nabla \times Q
	+\frac{\alpha}{2} \|Q\|_F^2-\frac{\beta}{3} (Q^2,Q)_F+\frac{\gamma}{4} \|Q\|_F^4\right] d^3r,\label{free_energy}
\end{alignat}
where $\|\nabla Q\|_F^2=\sum_{i,j,k=1}^3 (\partial_k Q_{ij})^2, Q: \nabla \times Q=\sum_{i,j,k=1}^3 \varepsilon_{ikl}Q_{ij}\partial_k Q_{lj}, \|Q\|_F^2=(Q,Q)_F=\sum_{i,j=1}^3 Q_{ij}^2,$
 $\varepsilon_{ikl}$ is the Levi-Civita symbol and $(\cdot,\cdot)_F$ denotes the Frobenius inner product.
Here $\Omega$ is a bounded domain in $\mathbb{R}^{3}$ with smooth boundary, and $Q$ is a symmetric traceless $3\times 3$ matrix.
The parameters $L_1$ and $L_4$ are elastic constants, which satisfy $L_1>0,L_1+L_4>0$ \cite{ravnik2009landau,wang2021modelling,deGennes1974}.
$\alpha =  a(\theta - \theta^*) = a \Delta \theta$ where $a > 0$, $\theta^*$ is the fixed temperature  of the system, and $\Delta \theta$ is the temperature difference from the transition temperature \cite{mottramintroduction}.
 $\beta$ and $\gamma>0$ are material constants that characterize the elastic properties of the liquid crystal \cite{nguyen2013refined}.

The gradient flow equation is derived from the free energy functional \eqref{free_energy} and describes the time evolution of $Q$ in a chiral liquid crystal. From \cite{mottramintroduction,wang2021modelling} the equation subject to periodic boundary condition and initial condition can be expressed as follows:
\begin{equation}\label{1.6}
	\begin{aligned}
			\partial_t Q &= L_1 \Delta Q -\frac{L_4}{2} (\nabla \times Q+(\nabla \times Q)^T)-  \alpha Q +\beta \left( Q^2 - \frac{1}{3}\text{trace}(Q^2)I\right) -\gamma \text{trace}(Q^2)Q, \\
& Q(0,\boldsymbol{x})=Q_0(\boldsymbol{x})  \qquad {\rm in}~  \Omega_0=\Omega\times\{t=0\},\\&
Q(t,\cdot)~\text{is } \Omega ~\text{periodic},  \qquad t \in [0,T],
	\end{aligned}
\end{equation}
where $-\frac{1}{3}\beta \text{trace}(Q^2)I$ accounts for the traceless constraint of the Q-tensor.

 Building upon the Landau-de Gennes Q-tensor framework, extensive theoretical and numerical studies have elucidated the formation mechanisms, equilibrium structures, and field-induced transformations of chiral blue phases, including the emergence of cubic BPI and BPII lattices \cite{RevModPhys.61.385,grebel1983landau,dupuis2005numerical,alexander2009numerical}, their stabilization under external electric fields \cite{kitzerow1991effect,chen2013electric}, and the amorphous defect networks characteristic of BP III \cite{henrich2011structure,gandhi2017unraveling,yoshizawa2024amorphous}.
 Especially, Hicks and Walker \cite{hicks2024modelling} presents a fully implicit, weighted gradient flow finite-element approach for the Landau-de Gennes cholesteric model,   accurately capturing the rich chiral structures observed experimentally.
Nevertheless, the majority of studies have been restricted to experimental or phenomenological simulations, lacking rigorous numerical formulations and systematic analyses of the underlying computational methods.


The numerical method we will adopt in this work is the exponential time differencing (ETD) method,  which be designed to  solve the stiff term exactly and  satisfy energy dissipation and maximum bound principle (MBP) \cite{beylkin1998new,cox2002exponential,du2019,du2021}.
The ETD method has been successfully applied to a wide range of problems, including  the Allen-Cahn equations \cite{du2019,du2021,jiang2022unconditionally}, satisfing MBP and energy dissipation unconditionally, the Cahn-Hilliard equations  \cite{li2019convergence, zhou2023energy,ZHANG2024108974,zhang2025convergence}, coupled with finite difference or Fourier spectral method in space and satisfying energy dissipation and mass
conservation unconditionally,
the epitaxial growth model \cite{ju2018energy} and matrix-valued Allen-Cahn equation \cite{du2021,hu2025maximumboundprincipleqtensor,liu2024maximum}.
{The high efficiency of the ETD type algorithm primarily relies on the fast diagonalization of discrete differential operators, such as Laplacian operator with suitable boundary conditions, but the curl operator in BPs model will bring significant challenge. Inspired by \cite{RevModPhys.61.385,bahr2001chirality,dupuis2005numerical}, we construct a Helmholtz basis in Fourier spectral space to diagonalize Laplace and curl operator simultaneously.}

In this work, we develop and analyze first and second order exponential time differencing  numerical schemes for the gradient flow of the  chiral liquid crystal blue phases, and give the fully discrete schemes by using the Fourier spectral method in space. 
We rigorously prove that the proposed schemes preserve the maximum bound principle  and energy dissipation unconditionally at the semi-discrete level and derive the error estimates in $L^2$ and $L^\infty$ norm.

 The rest of this paper is organized as follows. In Section \ref{section2}, we introduce some notations and present the numerical schemes. In Section \ref{section3}, we prove that the proposed schemes preserve the maximum bound principle  and energy dissipation unconditionally at the semi-discrete level. In Section \ref{section4}, we give the fully discrete schemes by using the Fourier spectral method in space and propose a novel matrix-form Helmholtz basis transformation method to diagonalize the combined operator of the Laplacian and the curl operator. In Section \ref{section5}, we derive the error estimates in $L^2$ and $L^\infty$ norm for the fully discrete schemes. In Section \ref{section6}, we present various numerical experiments to validate the theoretical results and demonstrate the effectiveness of the proposed methods in simulating the dynamics of blue phases in chiral liquid crystals. Finally, we conclude this paper in Section \ref{section7}.

\section{Exponential time differencing Runge-Kutta schemes}\label{section2}~\\
Based on the symmetric and traceless properties of the Q-tensor,  we define by $\mathbb{R}_{s,0}^{3 \times 3}$ the set of Q-tensors
\begin{equation}
	\mathbb{R}_{s,0}^{3 \times 3} \triangleq\left\{Q \in \mathbb{R}^{3 \times 3}\left\lvert\,Q_{i j}=Q_{j i},  Q_{i i}=0 \right.\right\}.
\end{equation}
The space $\mathcal{Z}=L^{\infty}\left(\bar{\Omega} ; \mathbb{R}_{s,0}^{3 \times 3}\right)$ is  the set of essentially bounded tensor functions  equipped with the standard $L^\infty$ norm:
\begin{align*}
	\|Q\|_\mathcal{Z}&=\max_{\boldsymbol{x} \in \bar{\Omega}}\left( \| Q(\boldsymbol{x})\|_F^2\right)^\frac{1}{2}.
\end{align*}
The sobolev space $\mathcal{X}=H^{1}\left(\bar{\Omega} ; \mathbb{R}_{s,0}^{3 \times 3}\right)$ is the space of square-integrable tensor functions with square-integrable first-order weak derivatives. We equip the space $\mathcal{X}$ with the standard $L^2$ inner product, denoted by $(\cdot, \cdot)_\mathcal{X}$, which for any two functions $P, Q \in \mathcal{X}$ is defined as:
\begin{equation}
    (P, Q)_\mathcal{X} := \int_{\Omega} P(\boldsymbol{x}) : Q(\boldsymbol{x}) \, d\boldsymbol{x}.
    \no
\end{equation}
The corresponding induced $L^2$-norm, denoted by $\|\cdot\|_\mathcal{X}$, and $H^1$-norm, denoted by $\|\cdot\|_{H^1}$ are  given by:
\begin{align*}
\|Q\|_\mathcal{X}&= (Q,Q)_\mathcal{X}^\frac{1}{2}:=\left(\int_{\Omega}\|Q(\boldsymbol{x})\|_F^{2} d\Omega\right)^{\frac{1}{2}},\\
\|Q\|_{H^1}&=\left(\|\nabla Q\|_{L^2}^2+\|Q\|_{L^2}^2\right)^{\frac{1}{2}}.
\end{align*}

For  constructing the ETD schemes and the following theoretical analysis, we introduce two stability positive constants $\kappa_1$ and $\kappa_2$ in  \eqref{1.6}.
Then  the right-hand side of \eqref{1.6} can be decomposed into a linear part $\mL Q$ and a nonlinear part $\mN (Q)$ as follows:
\begin{equation}\label{l_decompose}
    \begin{aligned}
        	\mL Q &= L_1\Delta Q - \frac{L_4}{2} (\nabla \times Q+(\nabla \times Q)^T)- \kappa_1 Q-\kappa_2 Q,\\
	\mN (Q) &= \kappa_1 Q +\kappa_2 Q+ f(Q),\quad
	f(Q)= -  \alpha Q +\beta \left( Q^2 - \frac{1}{3}\text{trace}(Q^2)I\right) -\gamma \text{trace}(Q^2)Q.
    \end{aligned}
\end{equation}
Using the above decomposition \eqref{l_decompose}, we can rewrite the equation  \eqref{1.6} into the following form:
\begin{equation}\label{1.8}
	\begin{aligned}
	&Q_t-\mL Q =\mN (Q), &\\
	&Q(0,\boldsymbol{x})=Q_0(\boldsymbol{x}),  & {\rm in}~  \Omega_0,\\
	 &Q(t,\cdot)~\text{is } \Omega ~\text{periodic},  & t \in [0,T].
	\end{aligned}
\end{equation}

 Define the time intervals as $t_m=m\tau$, $m\ge 0$, where $\tau>0$ is the time step size and set the numerical solution at $t_m$ as $Q_{m}$.
Based on the  Duhamel's formula \cite{Rousset2021},
   we can express the exact solution of \eqref{1.8} for an interval $[t_m, t_{m+1}]$  as follows:
   \begin{align}
	Q(t_{m+1})=e^{\tau \mL} Q(t_m)+\int_{0}^{\tau}e^{(\tau-\xi) \mL}\mN(Q(t_m+\xi))d\xi.\label{1.9a}
\end{align}
Using $ Q(t_m)$ to approximate $Q(t_m+\xi)$ in \eqref{1.9a}, we can derive the following  first-order explicit  ETD  scheme (ETD1) \cite{du2019,du2021,ju2018}
\begin{align}
		Q_{m+1}=e^{\tau \mL} Q_{m}+\int_{0}^{\tau}e^{(\tau-\xi) \mL}\mN(Q_{m})d\xi.\label{etd1}
\end{align}
Using $(1-\frac{s}{\tau})\mN(Q(t_m)) +\frac{s}{\tau} \mN(\widetilde{Q}(t_{m+1}))$ to approximate $Q(t_m+\xi)$, we get the following  second-order explicit  ETD Runge-Kutta  scheme (ETDRK2) \cite{du2019,du2021,ju2018}
\begin{equation}\label{etd2}
	\begin{aligned}
		\widetilde{Q}_{m+1}&=e^{\tau \mL} Q_{m}+\int_{0}^{\tau}e^{(\tau-\xi) \mL}\mN(Q_{m})d\xi\\
		Q_{m+1}&=e^{\tau \mL} Q_{m}+\int_{0}^{\tau}e^{(\tau-\xi) \mL}((1-\frac{s}{\tau})\mN(Q(t_m)) +\frac{s}{\tau} \mN(\widetilde{Q}(t_{m+1})))d\xi.
	\end{aligned}
\end{equation}

To express the schemes more concisely, we can introduce some $\varphi$-functions \cite{du2021} defined as
\begin{align}
    \varphi_0(z) = e^z, \quad \varphi_1(z) = \frac{e^z - 1}{z}, \quad \varphi_2(z) = \frac{e^z - 1 - z}{z^2}, \quad z \neq 0.\label{var_functions}
\end{align}
Then we can rewrite the schemes \eqref{etd1} and  \eqref{etd2} respectively as
\begin{align*}
		{\rm ETD1}: Q_{m+1}&=\varphi_0(\tau \mL) Q_{m}+\varphi_1(\tau \mL)\mN(Q_{m}),\\
		{\rm ETDRK2}:\widetilde{Q}_{m+1}&=\varphi_0(\tau \mL) Q_{m}+\tau \varphi_1(\tau \mL)\mN(Q_{m})\\
		Q_{m+1}&= \widetilde{Q}_m+\tau\varphi_2(\tau \mL)(\mN(\widetilde{Q}_{m+1})-\mN(Q_{m})).
\end{align*}
We further define the linear and nonlinear operators associated with each stabilization parameter $\kappa_i$ as:
\begin{align}
	\mL_{\kappa_i } Q = L_1\Delta Q - \frac{L_4}{2} (\nabla \times Q+(\nabla \times Q)^T)-\kappa_i Q,\quad
	\mN_{\kappa_i}(Q) = \kappa_i Q+ f(Q),\quad i=0,1,2,\label{Lkappa}
\end{align}
where $\kappa_0=0$ and the choice of $\kappa_1$ and $\kappa_2$ will be discussed in the next section.
\section{ Maximum bound principle and unconditional energy dissipation}\label{section3}
~\\
In this section, we  prove that the BPs equation \eqref{1.8} has a unique solution in $\mathcal{Z}$  and preserves the MBP property.
Then we show that the semi-discrete ETD schemes \eqref{etd1} and  \eqref{etd2} preserve the MBP and energy dissipation unconditionally.
    First, we show that the dissipation  of the  linear operator $\mL_{\kappa_2}$.
\begin{lemma}\label{3.1}
    Let the operator $\mL_{\kappa_2 }$ be defined in \eqref{Lkappa}.
	For $ W \in \mathcal{Z}$, when $L_1 \geq 0, \kappa_2 \geq \frac{ L_4^2}{2 L_1}$, there exists a positive constant $\lambda$ such that
	\begin{align}
	\lambda \Norm{W(\boldsymbol{x})}{\mathcal{Z}}&\leq \Norm{(\lambda I-\mL_{\kappa_2 })W(\boldsymbol{x})}{\mathcal{Z}} \label{1.17a},
	\end{align}
	Then
 the linear operator  $ \mL_{\kappa_2 }$  generates a contraction semigroup  $\left\{e^{t  \mL_{\kappa_2 }}\right\}_{t \geq 0},$ and for $t \geq 0$, it holds that
	\begin{align}
	\Norm{e^{t  \mL_{\kappa_2 }}W(\boldsymbol{x})}{\mathcal{Z}}&\leq \Norm{W(\boldsymbol{x})}{\mathcal{Z}}.\label{1.17b}
	\end{align}
\end{lemma}
\begin{proof}
	First, for any  $W \in \mathcal{Z}$, $W(\boldsymbol{x})=\{w_{ij}(\boldsymbol{x}),i=1,2,3,j=1,2,3 \}$, there exists $\boldsymbol{x_0} \in  \Omega$ (for the homogeneous Dirichlet boundary condition) or $\boldsymbol{x_0} \in \overline{\Omega}$ (for the periodic or homogeneous Neumann boundary condition) such that
\begin{align}
\Norm{W(\boldsymbol{x})}{\mathcal{Z}}=\max _{\boldsymbol{x} \in \bar{\Omega}}\Big|\sum_{i=1}^3 \sum_{j=1}^3 w_{ij}^2(\boldsymbol{x})\Big|^{\frac{1}{2}} =\| W(\boldsymbol{x_0})\|_F=\left(\sum_{i=1}^3 \sum_{j=1}^3 w_{ij}^2(\boldsymbol{x_0})\right)^{\frac{1}{2}}.
\end{align}

Since $\displaystyle\sum_{i=1}^3 \sum_{j=1}^3 w_{ij}^2(\boldsymbol{x_0})$ is a real scalar-valued function, and
  the compatibility of matrix norms,  at the point $\boldsymbol{x_0}$, we have
\begin{align}
\| \nabla \times W(\boldsymbol{x_0})\|_{F}^2 &\leq 2 \|\nabla W(\boldsymbol{x_0})\|_{F}^2,\\
	\| \nabla W(\boldsymbol{x_0})\|_{F}^2+W(\boldsymbol{x_0}):\Delta W(\boldsymbol{x_0}) &\leq  \Delta \| W(\boldsymbol{x_0})\|_{F}^2\leq 0.
\end{align}
Then using the Young's inequality and the above two inequalities, we have
\begin{align}
 \mL_{\kappa_2 } W(\boldsymbol{x_0}) : W(\boldsymbol{x_0})\leq&
L_1 W(\boldsymbol{x_0}):\Delta W(\boldsymbol{x_0})
-L_4 (\nabla \times W(\boldsymbol{x_0})): W(\boldsymbol{x_0})-{\kappa_2 } \|W(\boldsymbol{x_0})\|_{F}^2
\no\\\leq&
L_1 W(\boldsymbol{x_0}):\Delta W(\boldsymbol{x_0})+|L_4| \epsilon \|\nabla \times
 W(\boldsymbol{x_0})\|_{F}^2+ \frac{|L_4|}{4 \epsilon} |
 W(\boldsymbol{x_0})\|_{F}^2-{\kappa_2 } \|W(\boldsymbol{x_0})|_{F}^2
 \no\\\leq&
-L_1 \|\nabla W(\boldsymbol{x_0})|_{F}^2+2 |L_4|  \epsilon \|\nabla
 W(\boldsymbol{x_0})\|_{F}^2+ (\frac{|L_4|}{4 \epsilon}-{\kappa_2 }) \|
 W(\boldsymbol{x_0})\|_{F}^2
 \no\\\leq&
(-L_1+2 |L_4|  \epsilon ) \|\nabla W(\boldsymbol{x_0})|_{F}^2 + (\frac{|L_4|}{4 \epsilon}-{\kappa_2 }) \|
 W(\boldsymbol{x_0})\|_{F}^2.\label{t10}
\end{align}

When the right side of \eqref{t10} is less than or equal to zero, we can get
\begin{align*}
2|L_4|  \epsilon \leq L_1, \quad
(\frac{|L_4|}{4 \epsilon}-{\kappa_2 }) \leq 0.
\end{align*}
Through calculations, we can get
\begin{align*}
	 \epsilon \leq \frac{L_1}{2|L_4| }, \quad
 {\kappa_2 } \geq \frac{|L_4|}{4 \epsilon} \geq \frac{ L_4^2}{2 L_1}.
\end{align*}

That is when $L_1+L_4 \geq 0, \kappa_2 \geq \frac{ L_4^2}{2 L_1}$, we can choose a suitable $\epsilon$ such that
\begin{align*}
\mL_{\kappa_2 } W(\boldsymbol{x_0}) : W(\boldsymbol{x_0}) \leq 0.
\end{align*}
Then for any $\lambda  > 0$, we have
\begin{align}
	\lambda\| W(\boldsymbol{x_0})\|_F^2\leq &\lambda W^2(\boldsymbol{x_0})-\mL_{\kappa_2 } W(\boldsymbol{x_0}): W(\boldsymbol{x_0})\no\\=&
	 W(\boldsymbol{x_0}):(\lambda I-\mL_{\kappa_2 })W(\boldsymbol{x_0})
	\no\\\leq&\left( \|W(\boldsymbol{x_0})\|_F^2\right)^{\frac{1}{2}} \left( \|(\lambda I-\mL_{\kappa_2 })W(\boldsymbol{x_0})\|_F^2\right)^{\frac{1}{2}}
	\no\\=& \| W(\boldsymbol{x_0})\|_F \|(\lambda I-\mL_{\kappa_2 }) W(\boldsymbol{x_0})\|_F. \label{2.3}
\end{align}
Taking the supremum norm on both sides of \eqref{2.3}, we have
\begin{align*}
	\lambda \| W(\boldsymbol{x})\|_{\mathcal{Z}}\leq  \|(\lambda I-\mL_{\kappa_2 }) W(\boldsymbol{x})\|_{\mathcal{Z}}.
\end{align*}
Then we can use the Lumer-Phillips theorem to deduce the contraction property \eqref{1.17b} of the semigroup $\left\{e^{t  \mL_{\kappa_2 }}\right\}_{t \geq 0}$ \cite{du2021}.
\end{proof}

\begin{lemma}\label{3.2}
 Let  $a>0$ big enough, depending on the coefficients $(L_1,L_4,\alpha,\beta,\gamma)$ of the equation \eqref{1.6}, $b=\frac{\beta^2}{\gamma^2}-\frac{2(\alpha-{\kappa_2 })}{\gamma}$.  For any $Q \in \mathcal{Z} $ satisfing $\|Q\|^2_F\leq a^2$, for any $\boldsymbol{x} \in \bar{\Omega}$, when $\kappa_1 \geq\frac{(\kappa_2+C_f)^2}{\gamma (a^2-b)}, b<a^2$, we have
	   \begin{align}
	   &	\|\mN(Q)\|_F^2\leq \kappa_1^2 a^2,\no
	   \end{align}
   \end{lemma}
\begin{proof}
	 Using  the compatibility of matrix norms, we have
\begin{align}
	\fnorm{Q^2} \leq  \|Q\|_F^2, \quad
	 Q^2: Q\leq  \|Q\|_F^3,  \quad
	Q^2: Q^2 \leq \|Q\|_F^4.\label{t12}
\end{align}
Considering the term $\beta\|Q\|_F^3$, we have
\begin{align}
	\beta\|Q\|_F^3 - \frac{\beta^2}{2\gamma}\|Q\|_F^2 - \frac{\gamma}{2}\|Q\|_F^4= -(\gamma/2)\|Q\|_F^2 (x-\beta/\gamma)^2 \le 0.\label{t13a}
\end{align}
 Based on the  inequalities \eqref{t12}, we can estimate the term $\beta Q^2:Q$ as follows:
\begin{align}
	\beta Q^2:Q \le |\beta| \|Q\|_F^3
\le \frac{\beta^2}{2\gamma}\|Q\|_F^2 + \frac{\gamma}{2}\|Q\|_F^4.\label{t15}
\end{align}
Then substituting  \eqref{t15} into $\mN_{\kappa_2}(Q): Q$ and using the definition of $b$, we obtain
	\begin{align}
		\mN_{\kappa_2}(Q): Q=&-(\alpha-{\kappa_2 }) \|Q\|_F^2+\beta Q^2: Q-\gamma \|Q\|_F^4\no\\
			 &\leq (\frac{\beta^2}{2\gamma}-(\alpha-{\kappa_2 })) \|Q\|_F^2 - \frac{\gamma}{2}\|Q\|_F^4\no\\
		&=\frac{\gamma}{2}\|Q\|_F^2 (\frac{\beta^2}{\gamma^2}-\frac{2(\alpha-{\kappa_2 })}{\gamma}-\|Q\|_F^2)\no\\&
		=\frac{\gamma}{2}\|Q\|_F^2 (b-\|Q\|_F^2)
		\leq\frac{\gamma b^2}{8}.\label{1.15a}
	\end{align}

Based on the  inequalities \eqref{t12} and the assumption $\|Q\|_F\leq a$, there exists a positive constant ${C_f=|\alpha| +|\beta| a+ \gamma a^2+\gamma |b|}$  such that
\begin{align}
	\|f(Q)\|_F&\leq (|\alpha| +|\beta| a+ \gamma a^2 ) \|Q\|_F\leq {C_f}\|Q\|_F,\label{t13b}\\
	\|\mN_{\kappa_2}(Q)\|_F&=\|\kappa_2 Q+f(Q)\|_F\leq ({C_f}+\kappa_2)\|Q\|_F.\label{t13c}
\end{align}
	Substituting \eqref{1.15a} and \eqref{t13c} into $\|\mN(Q)\|_F^2$, we obtain
\begin{align}
		\|\mN(Q)\|_F^2
	=&\|\kappa_1 Q\|_F^2+2\kappa_1\mN_{\kappa_2}(Q): Q+\|\mN_{\kappa_2}(Q)\|_F^2 \no \\\leq
&\|\kappa_1 Q\|_F^2+\kappa_1\gamma\|Q\|_F^2 (b-\|Q\|_F^2)+(\kappa_2+C_f)^2\|Q\|_F^2.\label{t16}
\end{align}
Then we estimate the right-hand side of \eqref{t16}.
Let $ x=\|Q\|_F^2,x \in [0,a^2]$,  and define a function $g(x)$ as
\begin{align*}
	g(x)=\kappa_1^2x+ \gamma\kappa_1 x(b-x)+(\kappa_2+C_f)^2 x.
\end{align*}
 Thus we have $g'(x)=\kappa_1^2+ \gamma\kappa_1(b-2x)+(\kappa_2+C_f)^2$, which is a quadratic function with respect to $\kappa_1$.
And the discriminant of $g'(x)$ can be calculated as
\begin{align*}
	\operatorname{Disc}(g'(x))=\gamma^2 (b-2x)^2-4(\kappa_2+C_f)^2.
\end{align*}
Here we consider  letting $g'(x)\geq 0$ to estimate the maximum of $g(x)$ on the interval $[0,a^2]$.
Using the conditions $b<a^2$ and $0\leq x\leq a^2$, we can drive
\begin{align}
	 \frac{ 1}{2 }\gamma|b-2x|\leq \frac{ 1}{2 }\gamma (|b|+2|x|)\leq \gamma |b|+\gamma a^2\leq \kappa_2+C_f.\label{t17}
\end{align}
Then we can get $\operatorname{Disc}(g'(x))\leq 0$ and $g'(x)\geq 0$.
Thus for $g(x)$, we have
\begin{align}
	\frac{1}{a^2}g(x)\leq \frac{1}{a^2}g(a^2)\leq\kappa_1^2+ \gamma\kappa_1(b-a^2)+(\kappa_2+C_f)^2. \label{1.1b}
\end{align}
Substituting the conditions $b<a^2,\kappa_1 \geq\frac{(\kappa_2+C_f)^2}{\gamma (a^2-b)}$ into \eqref{t17}, we can get
\begin{align}
	\kappa_1^2+ \gamma\kappa_1(b-a^2)+(\kappa_2+C_f)^2\leq \kappa_1^2.\label{t18}
\end{align}
Combining \eqref{t16} and \eqref{1.1b}, we can get
\begin{align*}
	\|\mN(Q)\|_F^2\leq g(x)\leq \kappa_1^2 a^2,
\end{align*}
which completes the proof.
\end{proof}
\subsection{Continuous maximum bound principle}~
\begin{theorem}[Existence, uniqueness and MBP of the continuous equation \eqref{1.8}]\label{thm2}
	Suppose that $L_1 \geq 0, L_1+L_4 \geq 0, \kappa_1 \geq\frac{(\kappa_2+C_f)^2}{\gamma (a^2-b)},\kappa_2 \geq \frac{ L_4^2}{2 L_1}, b<a^2$ where $a,b$ defined in Lemma \ref{3.2}.
	Then for any $Q_0 \in  \mathcal{Z} $, the equation \eqref{1.8}   has a unique solution $Q(t,\boldsymbol{x}) \in C([0,T]; \mathcal{Z})$ satisfying $\|Q(t,\boldsymbol{x})\|^2_{\mathcal{Z}}\leq a^2$.
\end{theorem}
\begin{proof}
    We firstly prove that the solution $Q(t,\boldsymbol{x})$ satisfies MBP.
    Taking the supremum norm $\|\cdot\|_{\mathcal{Z}}$ on both sides of \eqref{1.9a} for the interval $[0,t]$, and using Lemma \ref{3.1} and \ref{3.2}, we obtain
    \begin{align*}
        \|Q(t_{})\|_\mathcal{Z} &= \left\|e^{t\mL} Q(0)+\int_{0}^{t}e^{(t-\xi) \mL}\mN(Q(\xi))d\xi\right\|_\mathcal{Z}\\
        &= \left\|e^{t(\mL_{\kappa_2}-\kappa_1)} Q(0)+\int_{0}^{t}e^{(t-\xi)(\mL_{\kappa_2}-\kappa_1)}\mN(Q(\xi))d\xi\right\|_\mathcal{Z}\\
        &\leq e^{-\kappa_1t}\|Q(0)\|_\mathcal{Z} + \int_{0}^{t} e^{-\kappa_1(t-\xi)} \|\mN(Q(\xi))\|_\mathcal{Z} d\xi\\
        &\leq e^{-\kappa_1t}a + \int_{0}^{t} e^{-\kappa_1(t-\xi)} \kappa_1 a d\xi\\
        &\leq e^{-\kappa_1t}a + \left(\frac{1}{\kappa_1} - \frac{1}{\kappa_1}e^{-\kappa_1t}\right)\kappa_1 a = a,
    \end{align*}
    which verify that $\|Q(t_{})\|_\mathcal{Z}^{2}\leq a^2$ holds for $t>0$.

    Using Lemma \ref{3.1} and \ref{3.2} and the above MBP analysis, we can see that the linear operator $\mL$ generates a strongly continuous semigroup on $\mathcal{Z}$ and the nonlinear operator $\mN$ satisfies the Lipschitz condition on $\mathcal{Z}$.
    Then according to the standard semigroup theory, we can prove the existence and uniqueness of the solution $Q(t,\boldsymbol{x}) \in C([0,T]; \mathcal{Z})$ of \eqref{1.8}. More details can referred to \cite{liu2024maximum}[Theorem 2.5].
\end{proof}
\subsection{Discrete maximum bound principle}~
\begin{theorem}[MBP of ETD schemes]\label{mbp}
For any $Q_0 \in \mathcal{Z} $ satisfing $\|Q_0\|^2_{\mathcal{Z}}\leq a^2$, when $\kappa_1 \geq\frac{(\kappa_2+C_f)^2}{\gamma (a^2-b)},\kappa_2 \geq \frac{ L_4^2}{2 L_1}, b< a^2$, the solution $Q_{m}$ generated by the ETD1 scheme \eqref{etd1} and  ETDRK2 scheme \eqref{etd2} satisfies
	\begin{align}
		\| Q_{m}\|_{\mathcal{Z}}^2\leq a^2, \qquad m \geq 0.
	\end{align}
\end{theorem}
\begin{proof}
	Suppose that $\|Q_{m}\|_F^2\leq a^2$, we will only need to show that $\|Q_{m+1}\|_F^2\leq a^2$ holds.
	Taking the supremum norm $\|\cdot\|_{\mathcal{Z}}$ on both sides of the ETD1 scheme \eqref{etd1}  and using Lemma \ref{3.1} and \ref{3.2}, we obtain
	\begin{align*}
\|Q_{m+1}\|_\mathcal{Z} &= \left\|e^{\tau\mL}Q_{m} + \int_0^\tau e^{(\tau-s)\tau\mL} \mathcal{N}[Q_{m}]\mathrm{d}s\right\|_\mathcal{Z}\\
&= \left\|e^{\tau(\mL_{\kappa_2}-\kappa_1)}Q_{m} + \int_0^\tau e^{(\tau-s)(\mL_{\kappa_2}-\kappa_1)} \mathcal{N}[Q_{m}]\mathrm{d}s\right\|_\mathcal{Z}\\
&\leq e^{-\kappa_1\tau}a + \left(\frac{1}{\kappa_1} - \frac{1}{\kappa_1}e^{-\kappa_1\tau}\right)\kappa_1 a = a.
	\end{align*}

	Taking  $\|\cdot\|_{\mathcal{Z}}$ on both sides of the ETDRK2 scheme \eqref{etd2} and using Lemma \ref{3.1} and \ref{3.2}, we obtain
	 \begin{align*}
\|Q_{m+1}\|_\mathcal{Z} &= \left\|e^{\tau(\mL_{\kappa_2}-\kappa_1)}Q_{m} + \int_0^\tau e^{(\tau-s)(\mL_{\kappa_2}-\kappa_1)} \left(\left(1-\frac{s}{\tau}\right)\mathcal{N}[Q_{m}] + \frac{s}{\tau}\mathcal{N}[\tilde{Q}_{m+1}]\right)\mathrm{d}s\right\|_\mathcal{Z}\\&
\leq e^{-\kappa_1\tau}a + \int_0^\tau e^{-\kappa_1(\tau-s)} \left(\left(1-\frac{s}{\tau}\right)\kappa_1 a + \frac{s}{\tau}\kappa_1 a\right)\mathrm{d}s\\&
\leq e^{-\kappa_1\tau}a + \left(\frac{1}{\kappa_1} - \frac{1}{\kappa_1}e^{-\kappa_1\tau}\right)\kappa_1 a = a,
	 \end{align*}
	  which verify that $\|Q_{m}\|_\mathcal{Z}^{2}\leq a^2$ holds for $m \geq 0$. The proof is complete.
\end{proof}
\begin{remark}
    Letting $|\cdot|_2$ denote the matrix 2-norm, we can derive
    \begin{align}
        |Q|_2 = \max_{i=1,2,3} \lambda_i (Q)\leq \|Q\|_F, \quad	\max_{\boldsymbol{x} \in \bar{\Omega}}|Q|_2\leq \Norm{Q}{\mathcal{Z}},\label{1.25ba}
    \end{align}
    where $\lambda_i (Q)$ denotes the eigenvalues of matrix $Q$,
which yields when the $\Norm{ }{\mathcal{Z}}$ norm of $Q$ satisfies the MBP, the matrix 2-norm of $Q$ also satisfies.
 \end{remark}
 \subsection{Discrete energy dissipation}~\\
In this section, we prove the unconditional energy dissipation of the ETD schemes \eqref{etd1} and \eqref{etd2}.
Using  the norm and inner product of $\mathcal{X}$, we  express the energy functional $E[Q]$ as follows:
\begin{alignat}{2}
	E[Q]= \frac{L_1}{2}\|\nabla Q\|_{\mathcal{X}}^2+\frac{L_4}{2}(Q,\nabla \times Q)_{\mathcal{X}}+\frac{\alpha}{2} \| Q\|_{\mathcal{X}}^2-\frac{\beta}{3} (Q,Q^2)_{\mathcal{X}}+\frac{\gamma}{4} \|Q\|_{\mathcal{X}}^4. \label{discrete_energy}
\end{alignat}
\begin{lemma}\label{3.4}
	 For any $Q_1, Q_2 \in \mathcal{Z}$ satisfing $\|Q_1\|_F\leq a, \|Q_2\|_F \leq a$, when $\kappa_1\geq |\alpha|+|\beta| a^2+2 \gamma a^2$, there holds
	 \begin{align*}
		 (F_b(Q_1) - F_b(Q_2), I)_F +  (Q_1 - Q_2, f(Q_2))_F  \leq \kappa_1\|Q_1 - Q_2\|_F^2,
	 \end{align*}
where $F_b(Q) = \frac{\alpha}{2} \| Q\|_{F}^2-\frac{\beta}{3} (Q,Q^2)+\frac{\gamma}{4} \|Q\|_{F}^4$.
\end{lemma}
\begin{proof}
	Based on the definition of $F_b(Q)$ and  $f(Q)$, we can derive the following equality:
	\begin{align}
		( F_b(Q_1) - F_b(Q_2), I )_F &= \frac{\alpha}{2} \| Q_1\|_{F}^2-\frac{\alpha}{2} \| Q_2\|_{F}^2-\frac{\beta}{3} (Q_1,Q_1^2)_F + \frac{\beta}{3} (Q_2,Q_2^2)_F + \frac{\gamma}{4} \|Q_1\|_{F}^4 - \frac{\gamma}{4} \|Q_2\|_{F}^4\no\\
		&= \frac{\alpha}{2} (\| Q_1\|_{F}^2-\| Q_2\|_{F}^2)-\frac{\beta}{3} ((Q_1-Q_2),Q_1^2+(Q_1+Q_2)Q_2)_F + \frac{\gamma}{4} (\|Q_1\|_{F}^4-\|Q_2\|_{F}^4)\no\\
		&= \alpha (Q_1,Q_1-Q_2)_F-\frac{\beta}{3} ((Q_1-Q_2),Q_1^2+(Q_1+Q_2)Q_2)_F \no\\
		&\quad+ \frac{\gamma}{2} (\|Q_1\|_{F}^2+\|Q_2\|_{F}^2)(Q_1,Q_1-Q_2)_F,\label{t19}\\
		(Q_1 - Q_2, f(Q_2))_F &= -\alpha (Q_1 - Q_2, Q_2)_F + \beta (Q_1 - Q_2, Q_2^2 - \frac{1}{3}\text{tr}(Q_2^2)I)_F + \gamma \text{tr}(Q_2^2)(Q_1 - Q_2, Q_2)_F\no\\
		&= -\alpha (Q_1 - Q_2, Q_2)_F + \beta (Q_1 - Q_2, Q_2^2)_F  -\gamma \text{tr}(Q_2^2)(Q_1 - Q_2, Q_2)_F.\label{t20}
	\end{align}
	 Using $\kappa_1 \geq |\alpha|+|\beta| a^2+2 \gamma a^2$ and Lemma \ref{mbp}, and adding \eqref{t19}-\eqref{t20} together, we can obtain
	\begin{align*}
		&\quad( F_b(Q_1) - F_b(Q_2), I )_F +  (Q_1 - Q_2, f(Q_2))_F \\
		&= \alpha (Q_1 - Q_2,Q_1 - Q_2)_F +\frac{|\beta|}{3}(Q_1 - Q_2, Q_1^2 + Q_2^2+Q_1Q_2-3Q_2^2)_F\\
		&\quad+\frac{\gamma}{2} (\|Q_1\|_{F}^2Q_1+\|Q_2\|_{F}^2Q_1-2\|Q_2\|_{F}^2Q_2,Q_1-Q_2)_F \\
		&\leq ((Q_1 - Q_2)(Q_1 - Q_2), \alpha I+\frac{|\beta|}{3}(Q_1 +2 Q_2)_F+\frac{\gamma}{2}(\|Q_2\|_{F}^2+\|Q_1\|_{F}^2+2Q_1Q_2))_F	\\
		&\leq (|\alpha|+|\beta| a^2+2\gamma a^2) \|Q_1-Q_2\|_{F}^2
		\leq {\kappa_1} \|Q_1-Q_2\|_{F}^2,
		\end{align*}
		which completes the proof.
\end{proof}
		\begin{lemma}\label{3.5}
	 For any $Q_1, Q_2 \in \mathcal{X}, L_1>0$, when $\kappa_2\geq \frac{L_4^2}{2L_1}$, there holds
	 \begin{align}
		 (F_e(Q_1) - F_e(Q_2), I)_{\mathcal{X}} +  (Q_1 - Q_2, \mL_0(Q_1))_{\mathcal{X}}  \leq \kappa_2 \|Q_1-Q_2\|_{\mathcal{X}}^2,\label{t21}
	 \end{align}
	 where $F_e(Q) =\frac{L_1}{2}\|\nabla Q\|_{\mathcal{X}}^2+\frac{L_4}{2} \innerx{Q}{\nabla \times Q}$,  $\mL_0(Q) =L_1 \Delta Q- \frac{L_4}{2} (\nabla \times Q + (\nabla \times Q)^T)$.
\end{lemma}
\begin{proof}
     Let $\delta Q = Q_1 - Q_2$.
Based on the definition of $F_e(Q)$ and  $\mL_0(Q)$, we can derive the following equality of the left-hand side of \eqref{t21}:
\begin{align}
&\quad ( F_e(Q_1) - F_e(Q_2), I )_{\mathcal{X}} +  (Q_1 - Q_2, \mL_0(Q_1))_{\mathcal{X}} \no\\
&= \frac{L_1}{2} \left( \|\nabla Q_1\|_{\mathcal{X}}^2 - \|\nabla Q_2\|_{\mathcal{X}}^2 \right) + \frac{L_4}{2} \left( (Q_1,\nabla \times Q_1)_{\mathcal{X}} - (Q_2,\nabla \times Q_2)_{\mathcal{X}} \right)\no \\
&\quad + L_1 (Q_1 - Q_2, \Delta Q_1)_{\mathcal{X}} - \frac{L_4}{2} (Q_1 - Q_2, \nabla \times Q_1 + (\nabla \times Q_1)^T)_{\mathcal{X}}\no\\
&:= I_1 + I_2 ,\label{t22}
\end{align}
where
\begin{align*}
I_1 &= \frac{L_1}{2} \left( \|\nabla Q_1\|_{\mathcal{X}}^2 - \|\nabla Q_2\|_{\mathcal{X}}^2 \right)+L_1 (\delta Q, \Delta Q_1)_{\mathcal{X}}, \\
I_2 &= \frac{L_4}{2} \left( (Q_1,\nabla \times Q_1)_{\mathcal{X}} - (Q_2,\nabla \times Q_2)_{\mathcal{X}}  - (\delta Q, \nabla \times Q_1 + (\nabla \times Q_1)^T)_{\mathcal{X}}\right).
\end{align*}
 For $I_1,$ using integration by parts and the periodic boundary condition, we have
\begin{align}
I_1 &= \frac{L_1}{2} \left(  (Q_2, \Delta Q_2)_{\mathcal{X}} -(Q_1, \Delta Q_1)_{\mathcal{X}} \right) +L_1 (\delta Q, \Delta Q_1)_{\mathcal{X}} \no \\
&= \frac{L_1}{2} \left( (Q_1 - \delta Q, \Delta Q_1 - \Delta\delta Q)_{\mathcal{X}} - (Q_1, \Delta Q_1)_{\mathcal{X}} \right) +L_1 (\delta Q, \Delta Q_1)_{\mathcal{X}}\no\\
&= \frac{L_1}{2} \left( -2(\delta Q, \Delta Q_1)_{\mathcal{X}} - \|\nabla\delta Q\|_{\mathcal{X}}^2 \right) +L_1 (\delta Q, \Delta Q_1)_{\mathcal{X}}=  - \frac{L_1}{2} \|\nabla\delta Q\|_{\mathcal{X}}^2.\label{t23a}
\end{align}
 For $I_2$, using integration by parts and the symmetry of $Q_1$ and $Q_2$, we have
\begin{align}
I_2 &= \frac{L_4}{2} \left( \innerx{Q_1}{\gradcurl{Q_1}}  - \innerx{Q_2}{\gradcurl{Q_2}} - \inner{\delta Q}{\gradcurl{Q_1} + (\gradcurl{Q_1})^T}_{\mathcal{X}}\right)  \no\\
%
%
&=\frac{L_4}{2} \left( \innerx{Q_1}{\gradcurl{Q_1}} - \innerx{Q_1-\delta Q}{\gradcurl{(Q_1-\delta Q)}}   - \inner{\delta Q}{\gradcurl{Q_1} + (\gradcurl{Q_1})^T}_{\mathcal{X}} \right) \no\\
%
%
&=\frac{L_4}{2} \left( (Q_1,\gradcurl{\delta Q})_{\mathcal{X}} - \inner{\delta Q}{(\gradcurl{Q_1})^T}_{\mathcal{X}} - \inner{\delta Q}{\gradcurl{\delta Q}}_{\mathcal{X}}\right)= - \frac{L_4}{2}\inner{\delta Q}{\gradcurl{\delta Q}}_{\mathcal{X}}.\label{t23b}
\end{align}
Based on the vector-calculus identity, we have
\begin{align}
	\|\nabla  \delta Q\|_{\mathcal{X}}^2=\|\nabla \times \delta Q\|_{\mathcal{X}}^2+\|\nabla \cdot \delta Q\|_{\mathcal{X}}^2.\label{t24}
\end{align}
Using Young's inequality and the vector-calculus identity \eqref{t24}, for $I_2$ we can derive
\begin{align}
 I_2
= -\frac{L_4}{2}  ( \delta Q, \nabla\times \delta Q )_{\mathcal{X}}\leq \frac{|L_4| \epsilon}{4}\|\nabla \times \delta Q\|_{\mathcal{X}}^2+\frac{|L_4| }{4\epsilon} \| \delta Q\|_{\mathcal{X}}^2\leq \frac{|L_4| \epsilon}{4}\|\nabla \delta Q\|_{\mathcal{X}}^2+\frac{|L_4| }{4\epsilon} \| \delta Q\|_{\mathcal{X}}^2.\label{t23}
\end{align}
Substituting \eqref{t23a} and \eqref{t23}  into \eqref{t22}, we obtain
\begin{align*}
	(F_e(Q_1) - F_e(Q_2), I)_{\mathcal{X}} +  (\delta Q, \mL_0(Q_2))_{\mathcal{X}}
	\leq \frac{|L_4|}{4\epsilon}\| \delta Q\|_{\mathcal{X}}^2+
	  (\frac{-L_1}{2} +\frac{|L_4| \epsilon}{4})\|\nabla  \delta Q\|_{\mathcal{X}}^2.
\end{align*}
Letting $\frac{-L_1}{2} +\frac{|L_4| \epsilon}{4}\leq 0$ and using $\kappa_2 \geq \frac{L_4^2}{2L_1}$, we can choose a constant $\epsilon$ such that
\begin{align*}
	\epsilon \leq \frac{2L_1}{|L_4|},\quad
 \frac{L_4^2}{8L_1} \leq \frac{|L_4|}{4\epsilon} \leq \kappa_2,
\end{align*}
 which completes the proof.
\end{proof}
	\begin{theorem}[Unconditional Energy Dissipation]\label{energy_stability}
	For any $Q_0 \in \mathcal{X}\cap \mathcal{Z}$ satisfing $\|Q_0\|^2_{\mathcal{Z}}\leq a^2$,   let $\{Q_m\}_{m\ge 0}$ be the numerical solution generated by  the ETD1 scheme \eqref{etd1} and the ETDRK2 scheme \eqref{etd2}. When $\kappa_1\geq \max\{\frac{(\kappa_2+C_f)^2}{\gamma (a^2-b)},|\alpha|+|\beta| a^2+2 \gamma a^2\}$ and $\kappa_2\geq \frac{L_4^2}{2L_1}$, for $m \geq 0$, we have  the energy dissipation law:
	\begin{align}
		E(Q_{m+1})\leq E(Q_{m}).\no
	\end{align}
	\end{theorem}
	 \begin{proof}
	 We first consider the ETD1 scheme \eqref{etd1}. Based on the  Lemma \ref{3.4} and Lemma \ref{3.5}, we can derive
\begin{align*}
 ( F_b(Q_{m+1}) - F_b(Q_{m}), I )_{\mathcal{X}}
&\le  ( Q_{m+1}-Q_{m}, -f(Q_{m}) )_{\mathcal{X}} + \kappa_1  \|Q_{m+1}-Q_{m}\|_\mathcal{X}^2  \\
&= -  ( Q_{m+1}-Q_{m}, \mN[Q_{m}] )_{\mathcal{X}} +  ( Q_{m+1}-Q_{m}, (\kappa_1+\kappa_2) Q_{m} )_{\mathcal{X}} + \kappa_1  \|Q_{m+1}-Q_{m}\|_\mathcal{X}^2,\\
(F_e(Q_{m+1}) - F_e(Q_{m}), I)_{\mathcal{X}} &\leq  (Q_{m+1}-Q_{m}, -\mL_0(Q_{m}))_{\mathcal{X}} +\kappa_2 \| Q_{m+1}-Q_{m}\|_{\mathcal{X}}^2  \\
&=  ( Q_{m+1}-Q_{m}, L Q_{m+1} )_{\mathcal{X}} -  ( Q_{m+1}-Q_{m}, (\kappa_1+\kappa_2) Q_{m+1} )_{\mathcal{X}} +\kappa_2\| Q_{m+1}-Q_{m}\|_{\mathcal{X}}^2,
\end{align*}
where $L=-\mL=-\mL_0+\kappa_1+\kappa_2$.
Adding the right-hand sides of the above two inequalities and subtracting  $\mN[Q_{m}]=L(\mathcal{I} - e^{-\tau L})^{-1} (Q_{m+1} - e^{-\tau L} Q_{m})$, we obtain
\begin{align}
    \innerx{Q_{m+1} - Q_{m}}{LQ_{m+1} -  \N[Q_{m}]} &= \innerx{Q_{m+1} - Q_{m}}{LQ_{m+1} - L(I - e^{-\tau L})^{-1}(Q_{m+1} - e^{-\tau L}Q_{m})} \nonumber \\
    &= \innerx{Q_{m+1} - Q_{m}}{L (I - e^{-\tau L})^{-1} \left(  - e^{-\tau L}Q_{m+1}  + e^{-\tau L}Q_{m} \right)} \nonumber \\
	&= \innerx{Q_{m+1} - Q_{m}}{-L\frac{e^{-\tau L}}{I - e^{-\tau L}}(Q_{m+1} - Q_{m})}
	    \nonumber \\
	&= \innerx{Q_{m+1} - Q_{m}}{\frac{-L}{e^{\tau L}-I}(Q_{m+1} - Q_{m})}
	    \nonumber
    \label{eq:intermediate_form}
\end{align}
Based on the definition of $E(Q)$ in \eqref{discrete_energy} and the above inequalities, we can obtain
\begin{align*}
		E(Q_{m+1}) - E(Q_{m})   &= ( F_b(Q_{m+1}) - F_b(Q_{m}), I )_{\mathcal{X}} + (F_e(Q_{m+1}) - F_e(Q_{m}), I)_{\mathcal{X}} \\
	     &\leq  ( Q_{m+1} - Q_{m}, \Delta_1 (Q_{m+1} - Q_{m}) )_{\mathcal{X}},
\end{align*}
where $\Delta_1 = \frac{-L}{e^{\tau L}-I}$.
Letting  $y_1(x) =  \frac{x}{e^{x}-1} $ for all $x \in \mathbb{R}$ (here we make a convention that $y_1(0) = 1$).
Then we know that $y_1(x)\geq 0$
and $\Delta_1 = -\frac{1}{\tau} y_1(-\tau L)$, we obtain that the operator $\Delta_1$ is non-positive definite which indicates that the energy is decreasing from $Q_{m}$ to $Q_{m+1}$. Thus, we have proved that the ETD1 scheme \eqref{etd1} satisfies the energy dissipation law.

Consider the ETDRK2 scheme \eqref{etd2}. Let $V = \widetilde{Q}^{n+1}$  obtained from the ETD1 scheme. The same as  the previous analysis, we can obtain
\begin{equation}
	E(V) - E(Q_{m}) \leq  ( V - Q_{m}, \Delta_1(V - Q_{m}) )_{\mathcal{X}}.\label{m3.19}
\end{equation}
Now $Q_{m+1}$ is obtained from the ETDRK2 scheme \eqref{etd2}.  Similar to the previous analysis, we can compute the energy difference between $Q_{m+1}$ and $V$:
\begin{align} E(Q_{m+1}) - E(V) &\leq  ( Q_{m+1} - V, LQ_{m+1} - \mathcal{N}[V] )_{\mathcal{X}} \no\\ &=  ( Q_{m+1} - V, LQ_{m+1} - \mathcal{N}[Q_{m}] - \tau (e^{-\tau L} - \mathcal{I} + \tau L^2)^{-1} L^2 (Q_{m+1} - V) )_{\mathcal{X}} \no\\
    &=  ( Q_{m+1} - V, \Delta_1(V - Q_{m})_{\mathcal{X}} + (L - \tau (e^{-\tau L} - \mathcal{I} + \tau L^2)^{-1} L^2) (Q_{m+1} - V) )_{\mathcal{X}} \no\\ &=  ( Q_{m+1} - V, \Delta_1(V - Q_{m}) )_{\mathcal{X}} +  ( Q_{m+1} - V, \Delta_2(Q_{m+1} - V) )_{\mathcal{X}}.\label{m3.20}
\end{align}
where $\Delta_2 = L - \tau (e^{-\tau L} - \mathcal{I} + \tau L)^{-1} L^2$. Combing \eqref{m3.19} with \eqref{m3.20}, we can obtain
\begin{align} E(Q_{m+1}) - E(Q_{m}) &\leq  ( V - Q_{m}, \Delta_1(V - Q_{m}) )_{\mathcal{X}} +  ( Q_{m+1} - V, \Delta_1(V - Q_{m}) )_{\mathcal{X}} \no\\ &\quad +  ( Q_{m+1} - V, \Delta_2(Q_{m+1} - V) )_{\mathcal{X}} \no\\ &= \frac{1}{2}  ( V - Q_{m}, \Delta_1(V - Q_{m}) )_{\mathcal{X}} +  ( Q_{m+1} - V, \left(\Delta_2 - \frac{1}{2}\Delta_1\right)(Q_{m+1} - V) )_{\mathcal{X}} \no\\ &\quad + \frac{1}{2}  \left[ ( V - Q_{m}, \Delta_1(V - Q_{m}) )_{\mathcal{X}}+ 2 ( Q_{m+1} - V, \Delta_1(V - Q_{m}) )_{\mathcal{X}}\right. \no\\ &\quad \left. + ( Q_{m+1} - V, \Delta_1(Q_{m+1} - V) )_{\mathcal{X}}\right]  \quad  \no\\ &= \frac{1}{2}  ( V - Q_{m}, \Delta_1(V - Q_{m}) )_{\mathcal{X}} +  ( Q_{m+1} - V, \left(\Delta_2 - \frac{1}{2}\Delta_1\right)(Q_{m+1} - V) )_{\mathcal{X}} \no\\ &\quad + \frac{1}{2}  ( Q_{m+1} - Q_{m}, \Delta_1(Q_{m+1} - Q_{m}) )_{\mathcal{X}}.\label{m3.21} \end{align}
Since $\Delta_1$ is non-positive definite, both the first and third terms of \eqref{m3.21} are non-positive. For the second term of \eqref{m3.21}, it is already proved in \cite{fu2022energy} that
$$ y_2(x) = x + \frac{x^2}{e^x - 1 - x} y_1(x) = \frac{x(e^x(x-3)+2)}{2(e^x-1)(e^x-1-x)} \geq 0, \quad \forall x \in \mathbb{R}, \quad  $$
where we make a convention that $y_2(0) = \frac{3}{2}$. Thus, the operator $\Delta_2 - \frac{1}{2}\Delta_1 = -\frac{1}{\tau} y_2(-\tau L)$ is non-positive definite. The second term of \eqref{m3.21} is also non-positive. Therefore, we have $E(Q_{m+1}) - E(Q_{m}) \leq 0$.
 \end{proof}
\section{Fully discrete exponential time differencing schemes}\label{section4}
\subsection{Spatial discretization}~\\
In this section, we summarize some notations  for the spectral collocation approximations of some spatial operators in the three-dimensional space with $\Omega = (-X, X) \times (-Y, Y) \times (-Z, Z)$.
 Let $N_x$, $N_y$, $N_z$ be three positive even numbers, $h_x = 2X/N_x$, $h_y = 2Y/N_y$,  $h_z = 2Z/N_z$ be the uniform mesh sizes in each direction, and $h = \max\{h_x, h_y, h_z\}$.
 We define the index sets
\begin{align*}
	S_h &= \{(x_i, y_j, z_k)|x_i = -X + ih_x, y_j = -Y + jh_y, z_k = -Z + kh_z| 1 \le i \le N_x, 1 \le j \le N_y, 1 \le k \le N_z\},\\
	\hat{S}_h &= \{(\frac{\pi n_1 }{X},\frac{\pi n_2 }{Y},\frac{\pi n_3 }{Z})  | -\frac{N_x}{2} + 1 \le n_1 \le \frac{N_x}{2}, -\frac{N_y}{2} + 1 \le n_2 \le \frac{N_y}{2}, -\frac{N_z}{2} + 1 \le n_3 \le \frac{N_z}{2}\}.
\end{align*}
All of the periodic grid functions defined on $S_h$ are denoted by $\mathcal{M}_h$, that is,
\[
\mathcal{M}_h = \{\mathbf{M}(\boldsymbol{x}): S_h \to \mathbb{R}^{3\times 3} | \boldsymbol{x} \in S_h\}.
\]
For any $\mathbf{M},\mathbf{N} \in \mathcal{M}_h$, the discrete $L^2$ inner product $ \il{\cdot}{\cdot}$, discrete $L^2$ norm $\|\cdot\|_2$, and discrete $L^\infty$ norm $\|\cdot\|_\infty$ are respectively defined by
\begin{equation}\label{for1}
    \begin{aligned}
   \il{\mathbf{M}}{\mathbf{N}} &= h_x h_y h_z \sum_{\boldsymbol{x}\in S_h} (\mathbf{M}(\boldsymbol{x}):\mathbf{N}(\boldsymbol{x})),\quad
  \|\mathbf{M}\|_2 = \sqrt{\il{\mathbf{M}}{\mathbf{M}}}, \\  \|\mathbf{M}\|_\infty &= \max_{\boldsymbol{x}\in S_h} \|\mathbf{M}(\boldsymbol{x})\|_F.
    \end{aligned}
\end{equation}
Let $\widehat{\mathcal{M}}_h=\{\mathcal{F} f \mid f\in\mathcal{M}_h\}$ be the discrete Fourier space, where $\mathcal{F}$ is the discrete Fourier transform operator. For any $\widehat{A}_{\boldsymbol{k}}, \widehat{B}_{\boldsymbol{k}} \in \widehat{\mathcal{M}}_h$, the standard Hermitian inner product is denoted by
\begin{align*}
    \langle \widehat{A}_{\boldsymbol{k}}, \widehat{B}_{\boldsymbol{k}} \rangle_F &= \sum_{\boldsymbol{k} \in \hat{S}_h} \widehat{A}_{\boldsymbol{k}} : \overline{\widehat{B}_{\boldsymbol{k}}}, \quad \text{where } \overline{\widehat{B}_{\boldsymbol{k}}} \text{ is the complex conjugate of } \widehat{B}_{\boldsymbol{k}}.
\end{align*}
\subsection{Discrete  Laplace and curl operators}
Let $\boldsymbol{k} \in \hat{S}_h$ be the wave number vector and $\boldsymbol{x} \in S_h$ be the spatial variable.
For any $Q_h \in \mathcal{M}_h$, we denote by $\widehat{Q}_{\boldsymbol{k}}$ the discrete Fourier transform of $Q_h$, which is given by
\begin{align*}
	\widehat{Q}_{\boldsymbol{k}} = (\mathcal{F} Q_h)_{\boldsymbol{k}} = \sum_{\boldsymbol{x} \in S_h}  Q_h(\boldsymbol{x}) \, e^{-i\boldsymbol{k}\cdot\boldsymbol{x}}, \quad \forall \boldsymbol{k} \in \hat{S}_h
\end{align*}
and the inverse discrete Fourier transform is given by
\begin{align*}
	 Q_h(\boldsymbol{x})  =\mathcal{F}^{-1} \widehat{Q}_{\boldsymbol{k}} = \frac{1}{N_x N_y N_z} \sum_{\boldsymbol{k} \in \hat{S}_h} \widehat{Q}_{\boldsymbol{k}} e^{i\boldsymbol{k}\cdot\boldsymbol{x}}, \quad  \boldsymbol{x} \in S_h.
\end{align*}
The discrete Laplace operator $\Delta_h=: \mathcal{M}_h \to \mathcal{M}_h$ is defined via its action in the Fourier space as:
\begin{align*}
	\mathcal{F}(\Delta_h  Q_h(\boldsymbol{x}))  = - |\boldsymbol{k}|^2\widehat{Q}_{\boldsymbol{k}}, \quad \boldsymbol{k} \in \hat{S}_h.
\end{align*}
Let the discrete curl operator, $\nabla_h \times: \mathcal{M}_h \to \mathcal{M}_h$, be defined row-wise via its symbol in the Fourier space.  Let $\widehat{\boldsymbol{q}}_{r,\boldsymbol{k}}$ for $r = 1, 2, 3$ be the Fourier transform of the $r$-th row vector field of $Q_h$. The action of the curl operator in Fourier space is given by:
\begin{align*}
    (\mathcal{F}(\nabla_h \times Q_h))_{\boldsymbol{k}}= i\boldsymbol{k} \times_r \widehat{Q}_{\boldsymbol{k}} =
    \begin{pmatrix}
        i\boldsymbol{k} \times \widehat{\boldsymbol{q}}_{1,\boldsymbol{k}} \\
        i\boldsymbol{k} \times \widehat{\boldsymbol{q}}_{2,\boldsymbol{k}} \\
        i\boldsymbol{k} \times \widehat{\boldsymbol{q}}_{3,\boldsymbol{k}}
    \end{pmatrix},
    \text{where} \times_r \text{denotes the row-wise cross product.}
\end{align*}
We then define the operator $\mathcal{C}_h: \mathcal{M}_h \to \mathcal{M}_h$ as the sum of the curl and its transpose and the Fourier multiplier corresponding to $\mathcal{C}_h$, denoted by $\widehat{\mathcal{C}}(\boldsymbol{k})$, acts on a matrix $\widehat{Q}_{\boldsymbol{k}} \in \mathbb{C}^{3\times3}$ as follows:
\begin{align*}
    \mathcal{C}_h Q_h := (\nabla_h \times Q_h) + (\nabla_h \times Q_h)^T,\quad\widehat{\mathcal{C}}(\boldsymbol{k})[\widehat{Q}_{\boldsymbol{k}}] := i\boldsymbol{k} \times_r \widehat{Q}_{\boldsymbol{k}} + (i\boldsymbol{k} \times_r \widehat{Q}_{\boldsymbol{k}})^T.
\end{align*}
The action of the full operator in Fourier space is thus
    $(\mathcal{F}(\mathcal{C}_h Q_h))_{\boldsymbol{k}} = \widehat{\mathcal{C}}(\boldsymbol{k})[\widehat{Q}_{\boldsymbol{k}}].$
We denote by $\mathcal{L}_h: \mathcal{M}_h \to \mathcal{M}_h$ and $\mathcal{L}_{h,\kappa_2}: \mathcal{M}_h \to \mathcal{M}_h$ the discrete linear operator defined by
\begin{align}
	\mathcal{L}_h  Q_h = L_1 \Delta_h  Q_h - \frac{L_4}{2} {\mathcal{C}}_h Q_h - \kappa_1  Q_h - \kappa_2  Q_h,~~
    	\mathcal{L}_{h,\kappa_2}  Q_h = L_1 \Delta_h  Q_h - \frac{L_4}{2} {\mathcal{C}}_h Q_h - \kappa_2  Q_h.\label{a5_1}
\end{align}
 The discrete nonlinear operator $\mathcal{N}_h: \mathcal{M}_h \to \mathcal{M}_h$ is defined by
\begin{align}
	\mathcal{N}_h( Q_h) = -\alpha  Q_h + \beta \left(  Q_h^2 - \frac{1}{3} \text{tr}( Q_h^2) I \right) - \gamma \text{tr}( Q_h^2)  Q_h + \kappa_1  Q_h + \kappa_2  Q_h.\label{nonlinear_operator}
\end{align}

Using the above notations, we can express the semi-discrete system of \eqref{1.8} as follows:
\begin{align}
	\frac{d Q_h}{dt} = \mathcal{L}_h  Q_h + \mathcal{N}_h( Q_h), \quad t > 0, \quad  Q_h(0,\boldsymbol{x}) =  Q_{0}(\boldsymbol{x}), \quad \forall \boldsymbol{x} \in S_h. \label{semi-discrete}
\end{align}
We denote by $\{Q_h^{m}\}_{m \geq 0}$ the numerical solution at $t_m = m \tau$ with the time step size $\tau > 0$. The fully discrete ETD1 scheme is given by
\begin{align}
	{Q}^{m+1}_{h} = \phi_0(\mathcal{L}_h \tau) {Q}^{m}_{h} + \tau \phi_1(\mathcal{L}_h \tau) {\mathcal{N}_h(Q_h^m)},\quad m \ge 0,\quad Q^0(\boldsymbol{x}) = Q_0(\boldsymbol{x})\quad \forall \boldsymbol{x} \in S_h. \label{etd1h}
\end{align}
 The fully discrete ETDRK2 scheme is given by
\begin{equation}\label{etd2h}
	\begin{aligned}
	\widetilde{Q}^{m+1}_{h} &= \phi_0(\mathcal{L}_h \tau) {Q}^{m}_{h} + \tau \phi_1(\mathcal{L}_h \tau) {\mathcal{N}_h(Q_h^m)} \\
	{Q}^{m+1}_{h} &=\widetilde{Q}^{m+1}_{h}+  \tau \phi_2(\mathcal{L}_h \tau) {(\mathcal{N}_h(\widetilde{Q}^{m+1}_{h})-\mathcal{N}_h(Q_h^m))}, \quad m \ge 0, \quad Q^0(\boldsymbol{x}) = Q_0(\boldsymbol{x})\quad \forall \boldsymbol{x} \in S_h.
	\end{aligned}
\end{equation}

The above ETD schemes can be efficiently implemented by using the fast Fourier transform (FFT) and its inverse transform. The computational cost of each time step is $\mathcal{O}(N^3 \log N)$.
    \subsection{Efficient implementations of the ETD schemes}
In this subsection, we will give the efficient implementations of the fully discrete ETD schemes   based on the discrete Fourier transform.  We denote by $\widehat{\mathcal{L}}_{\boldsymbol{k}},\widehat{\mathcal{L}}_{\boldsymbol{k},\kappa_2}$ respectively the Fourier symbol of the operator $\mathcal{L}_h,\mathcal{L}_{h,\kappa_2}$:
\begin{align}
	\widehat{\mathcal{L}}_{\boldsymbol{k}} = -L_1 |\boldsymbol{k}|^2 I - \frac{L_4}{2} \widehat{\mathcal{C}}(\boldsymbol{k}) - \kappa_1 I- \kappa_2 I,\quad \widehat{\mathcal{L}}_{\boldsymbol{k},\kappa_2} = -L_1 |\boldsymbol{k}|^2 I - \frac{L_4}{2} \widehat{\mathcal{C}}(\boldsymbol{k}) - \kappa_2 I. \label{fourier_symbol}
\end{align}
Then the key process of calculating $Q_{m+1}$ from the schemes  is the efficient implementation of the actions of the operator exponentials,
$\phi_\gamma(\mathcal{L}_h \tau) V$, $\gamma = 0, 1, 2$. Since $\mathcal{L}_h$ comes from the discretization of $\mathcal{L}$ with the periodic
boundary condition, the exponentials $\phi_\gamma(\mathcal{L}_h \tau)$ can be implemented by the 3D discrete
Fourier transform (DFT). However, the operator $\mathcal{L}_h$ cannot be diagonalized directly by the 3D DFT due to the presence of the discrete curl operator $\nabla_h \times$ and its transform. To address this challenge, we employ a spectral decomposition based on the helicity basis.
This basis provides a tensor-valued generalization of the Helmholtz decomposition, separating the tensor field into five independent modes, each with a definite helicity ($\lambda = 0, \pm 1, \pm 2$). For brevity and to emphasize its properties, we will refer to our constructed basis as the Helmholtz Symmetric Trace-Free \textbf{(HSTF)} Basis.

Here we will give a detailed description of how to construct the HSTF Basis in the Fourier space.

\subsubsection{ Construction of the Helmholtz Symmetric Trace-Free  Basis}~\\
For any given wavevector $\vect{k}\in \hat{S}_h \neq \vect{0}$, we first construct a local right-handed orthonormal coordinate system $\{\vu{e}_1, \vu{e}_2, \vu{k}\}$, where $\vu{k}= \frac{\vect{k}}{|\vect{k}|}$ is the unit vector aligned with $\vect{k}$. Then we choose a fixed, arbitrary vector $\vect{v}$ (e.g., $\vect{v}=(1,0,0)$) that is not parallel to $\vu{k}$. Using the Gram-Schmidt process, the first transverse vector $\vu{e}_1$ is obtained by projecting $\vect{v}$ onto the plane normal to $\vu{k}$ and normalizing the result:
\begin{equation}
    \vect{e}_1' = \vect{v} - (\vect{v} \cdot \vu{k}) \vu{k}, \quad \text{and} \quad \vu{e}_1 = \frac{\vect{e}_1'}{|\vect{e}_1'|}.
    \end{equation}
    If $\vu{k}$ is nearly parallel to $\vect{v}$, a different vector (e.g., $\vect{v}=(0,1,0)$) is chosen to avoid numerical instability.

  Using the cross product, the second transverse vector $\vu{e}_2$ is defined as:
    \begin{equation}
        \vu{e}_2 = \vu{k} \times \vu{e}_1,
    \end{equation}
	which is orthogonal to both $\vu{k}$ and $\vu{e}_1$ and completes the right-handed orthonormal basis.
	Then, using the vector triple product identity, we can obtain:
	\[
	\vu{k} \times\vu{e}_2 =  \vu{k} \times (\vu{k} \times \vu{e}_1) = (\vu{k} \cdot \vu{e}_1)\vu{k} - (\vu{k} \cdot \vu{k})\vu{e}_1 = -\vu{e}_1.
\]
It is clear that $\vu{e}_1$ and $\vu{e}_2$ are not still the eigenvectors of the curl operator. Finally, we make a rotation to construct the complex transverse basis vectors $\bm{m}_{\pm}(\hat{\bm{k}})$ as follows:
\begin{equation}
    \bm{m}_{\pm}(\hat{\bm{k}}) = \frac{1}{\sqrt{2}}(\hat{\bm{e}}_1 \pm i\hat{\bm{e}}_2),
\end{equation}
which satisfy
		$\bm{k} \times \bm{m}_{+} = -i|\bm{k}|\bm{m}_{+}, \bm{k} \times \bm{m}_{-} = i|\bm{k}|\bm{m}_{-}. $

 Through the dyadic products of the local basis vectors $\{\bm{m}_{+}, \bm{m}_{-}, \vu{k}\}$, we can obtain nine mutually orthogonal basis tensors:
$\{\bm{m}_{+} \otimes \bm{m}_{+}, \bm{m}_{+} \otimes \bm{m}_{-},\ldots, \vu{k} \otimes \vu{k}\}$, where $\otimes$ denotes the tensor product: $(\boldsymbol{a} \otimes \boldsymbol{b})_{ij} = a_i b_j$ for any vectors $\boldsymbol{a}, \boldsymbol{b} \in \mathbb{R}^3$.
However, we are specifically interested in the subspace of symmetric trace-free  tensors, which is 5-dimensional, satisfing the symmetry ($\tensor{Q} = \tensor{Q}^T$) and a zero trace ($\text{Tr}(\tensor{Q}) = 0$). Let $ \Gamma= \{0,+k,-k,+2k,-2k\}$.
Through linear combinations of the nine basis tensors, we can construct the HSTF Basis $\{\tensor{T}_j(\vu{k})\}_{j \in \Gamma}$ as follows:
    \begin{align}
        \tensor{T}_0(\vu{k}) &= \frac{1}{\sqrt{6}} \left( \vu{e}_1 \otimes \vu{e}_1 + \vu{e}_2 \otimes \vu{e}_2 - 2 \vu{k} \otimes \vu{k} \right) = \frac{1}{\sqrt{6}} (\tensor{I} - 3\vu{k}\otimes\vu{k}),
\\
        \tensor{T}_{\pm k}(\vu{k}) &= \frac{1}{\sqrt{2}} \left( \vect{m}_{\pm}(\vu{k}) \otimes \vu{k} + \vu{k} \otimes \vect{m}_{\pm}(\vu{k}) \right),
\\
        \tensor{T}_{\pm 2k}(\vu{k}) &= \vect{m}_{\pm}(\vu{k}) \otimes \vect{m}_{\pm}(\vu{k}).
    \end{align}
We can verify that the HSTF bases are symmetric, trace-free and mutually orthogonal:
\begin{align*}
	\tensor{T}_{  j}^T &= \tensor{T}_{  j}, \quad \text{Tr}(\tensor{T}_{  j}) = 0,\quad   \langle \tensor{T}_{  j}, \tensor{T}_{  l} \rangle_F=\delta_{jl}, \quad \forall j,l \in \Gamma.
\end{align*}
\subsubsection{Diagonalization of the discrete $\mL_N$ operator}~\\
Let's first consider the action on $\tensor{T}_{+2k} = \vect{m}_{+} \otimes \vect{m}_{+}$. Using the tensor identity $\boldsymbol{a} \timesr (\boldsymbol{b} \otimes \boldsymbol{c}) = (\boldsymbol{a} \times \boldsymbol{b}) \otimes \boldsymbol{c}$, we have:
\begin{align}
    \boldsymbol{k} \timesr \tensor{T}_{+2k} &= \boldsymbol{k} \timesr (\vect{m}_{+} \otimes \vect{m}_{+}) = (\boldsymbol{k} \times \vect{m}_{+}) \otimes \vect{m}_{+} \nonumber \\
    &= (-i|\boldsymbol{k}|\vect{m}_{+}) \otimes \vect{m}_{+} = -i|\boldsymbol{k}| \tensor{T}_{+2k}. \label{eq:curl_on_T2k}
\end{align}
Since $\tensor{T}_{+2k}$ is symmetric, $(\tensor{T}_{+2k})^T = \tensor{T}_{+2k}$. Substituting \eqref{eq:curl_on_T2k} into the operator symbol's definition:
\begin{align}
    \widehat{\mathcal{C}}(\boldsymbol{k})[\tensor{T}_{+2k}] &= i\left( (-i|\boldsymbol{k}|\tensor{T}_{+2k}) + (-i|\boldsymbol{k}|\tensor{T}_{+2k})^T \right) \nonumber \\
    &= i\left( -i|\boldsymbol{k}|\tensor{T}_{+2k} - i|\boldsymbol{k}|\tensor{T}_{+2k} \right) = 2|\boldsymbol{k}|\tensor{T}_{+2k}.
\end{align}
Similarly, for $\tensor{T}_{-2k} = \vect{m}_{-} \otimes \vect{m}_{-}$, we find $\boldsymbol{k} \timesr \tensor{T}_{-2k} = +i|\boldsymbol{k}| \tensor{T}_{-2k}$, which leads to:
\begin{equation}
    \widehat{\mathcal{C}}(\boldsymbol{k})[\tensor{T}_{-2k}] = -2|\boldsymbol{k}|\tensor{T}_{-2k}.
\end{equation}

Now consider the action on $\tensor{T}_{+k} = \frac{1}{\sqrt{2}}(\vect{m}_{+} \otimes \vu{k} + \vu{k} \otimes \vect{m}_{+})$. We analyze the two terms separately:
\begin{align}
    \boldsymbol{k} \timesr (\vect{m}_{+} \otimes \vu{k}) &= (\boldsymbol{k} \times \vect{m}_{+}) \otimes \vu{k} = -i|\boldsymbol{k}|(\vect{m}_{+} \otimes \vu{k}). \label{eq:curl_on_Tk}\\
    \boldsymbol{k} \timesr (\vu{k} \otimes \vect{m}_{+}) &= (\boldsymbol{k} \times \vu{k}) \otimes \vect{m}_{+} = \boldsymbol{0}.
\end{align}
The basis tensor $\tensor{T}_{+k}$ is symmetric by construction. Substituting \eqref{eq:curl_on_Tk} into the operator definition:
\begin{align}
    \widehat{\mathcal{C}}(\boldsymbol{k})[\tensor{T}_{+k}] &= i\left( \left(-\frac{i|\boldsymbol{k}|}{\sqrt{2}}(\vect{m}_{+} \otimes \vu{k})\right) + \left(-\frac{i|\boldsymbol{k}|}{\sqrt{2}}(\vect{m}_{+} \otimes \vu{k})\right)^T \right) \nonumber \\
    &= \frac{|\boldsymbol{k}|}{\sqrt{2}} \left( \vect{m}_{+} \otimes \vu{k} + \vu{k} \otimes \vect{m}_{+} \right) = |\boldsymbol{k}|\tensor{T}_{+k}.
\end{align}
By a similar calculation for $\tensor{T}_{-k}$, using $\boldsymbol{k} \times \vect{m}_{-} = +i|\boldsymbol{k}|\vect{m}_{-}$, we obtain:
\begin{equation}
    \widehat{\mathcal{C}}(\boldsymbol{k})[\tensor{T}_{-k}] = -|\boldsymbol{k}|\tensor{T}_{-k}.
\end{equation}
This completes the verification that the four complex HSTF basis tensors are indeed eigentensors of the curl operator symbol, with eigenvalues proportional to their helicity.
Then the eigenvalues of the operator $\widehat{\mathcal{L}}_{\boldsymbol{k}}$ with respect to the HSTF basis can be given by
\begin{equation}\label{eigenvalues}
    \begin{aligned}
	\lambda_{\boldsymbol{k}}^{(j)} &= -L_1 |\boldsymbol{k}|^2 +\frac{j L_4 }{2} |\boldsymbol{k}|- \kappa_1-\kappa_2, \quad j \in \Gamma.
    \end{aligned}
\end{equation}

For any given wavevector $\boldsymbol{k} \neq \boldsymbol{0}$,   any discrete Fourier mode $\widehat{Q}_{\boldsymbol{k}}$ can be expanded in the HSTF basis $\{\tensor{T}_0, \tensor{T}_{\pm k}, \tensor{T}_{\pm 2k}\}$ as follows:
\begin{equation*}
    \widehat{Q}_{\boldsymbol{k}} = \sum_{j \in \Gamma} c_j(\boldsymbol{k}) \tensor{T}_j(\vu{k}), \quad \text{where } c_j(\boldsymbol{k}) = (\widehat{Q}_{\boldsymbol{k}}, \tensor{T}_j(\vu{k}))_F.
\end{equation*}
Then the Fourier symbol $\widehat{\mathcal{L}}_{\boldsymbol{k}}$ of the operator $\mathcal{L}_h$ can be diagonalized in the HSTF basis as follows:
\begin{align*}
	\widehat{\mathcal{L}}_{\boldsymbol{k}} \widehat{Q}_{\boldsymbol{k}}= \sum_{j \in \Gamma} \lambda_{\boldsymbol{k}}^{(j)} c_j(\boldsymbol{k}) \tensor{T}_j(\vu{k}).
\end{align*}
\subsubsection{Dissipation of the discrete linear operator $\mathcal{L}_{h,\kappa_2}$}~\\
In this subsection, we will prove that the discrete linear operator $\mathcal{L}_{h,\kappa_2}$ defined in \eqref{a5_1} is dissipative under certain conditions, which is crucial for the stability analysis of the ETD schemes.
\begin{lemma}
\label{5_1}
Let $Q_h \in \mathcal{M}_h$ be any periodic grid function. Let $\mathcal{L}_{h,\kappa_2}$ be defined in \eqref{a5_1}. If $L_1 > 0$ and the stabilization parameter is chosen such that $\kappa_2 \geq \frac{L_4^2}{2L_1}$, then the operator $\mathcal{L}_{h,\kappa_2}$ is dissipative (negative semi-definite) in the discrete $L^2$ inner product, i.e.,
\begin{equation*}
   \il{\mathcal{L}_{h,\kappa_2}Q_h}{Q_h} \leq 0.
\end{equation*}
And the operator $\mathcal{L}_{h,\kappa_2}$ generates a contraction semigroup in the discrete $L^2$ norm, i.e.,
\begin{equation*}
    \|\mathrm{e}^{\mathcal{L}_{h,\kappa_2} t} Q_h\|_2 \leq \|Q_h\|_2, \quad \forall t \geq 0.
\end{equation*}
\end{lemma}

\begin{proof}
    Based on \eqref{eigenvalues} and the definition in \eqref{fourier_symbol}, we  have that the eigenvalues of the operator $\widehat{\mathcal{L}}_{\boldsymbol{k},\kappa_2}$, denoted by $\lambda_{\boldsymbol{k},\kappa_2}$, with respect to the HSTF basis are given by
    \begin{equation*}
        \lambda_{\boldsymbol{k},\kappa_2}^{(j)} = -L_1 |\boldsymbol{k}|^2 +\frac{j L_4 }{2} |\boldsymbol{k}|- \kappa_2, \quad j \in \Gamma.
    \end{equation*}
By applying the discrete Parseval's theorem, for a discrete wavevector $\boldsymbol{k} \in  \hat{S}_k$, we can express the discrete $L^2$ inner product in terms of the Discrete Fourier Transform (DFT) coefficients of the grid function $Q_h$:
\begin{equation*}
   \il{\mathcal{L}_{h,\kappa_2}Q_h}{Q_h} = C_N  \left\langle\widehat{(\mathcal{L}_{h,\kappa_2}Q_h)}_{\boldsymbol{k}} , \overline{\widehat{(Q_h)}_{\boldsymbol{k}}}\right\rangle_F =C_N\sum_{\boldsymbol{k} \in  \hat{S}_k} \left( \sum_{j \in \Gamma} \lambda_{\boldsymbol{k},\kappa_2}^{(j)}c_j(\boldsymbol{k}) \tensor{T}_j(\vu{k}) \right) : \left( \sum_{l \in \Gamma} \overline{c_l(\boldsymbol{k})} \overline{\tensor{T}_l(\vu{k})} \right),
\end{equation*}
where $C_N$ is a positive normalization constant.
Since the HSTF basis is orthonormal, and it consists of eigenvectors of $\hat{\mathcal{L}}_{\kappa_2}(\boldsymbol{k})$ with corresponding real eigenvalues $\lambda_{\boldsymbol{k}}^{(j)}$, the inner product term simplifies to:
\begin{align}
  \il{\mathcal{L}_{h,\kappa_2}Q_h}{Q_h}
    &= C_N\sum_{\boldsymbol{k} \in  \hat{S}_k}\left\langle \sum_{j \in \Gamma} c_j(\boldsymbol{k}) \lambda_{\boldsymbol{k},\kappa_2}^{(j)}\mathbf{T}_j(\boldsymbol{k}), \sum_{l \in \Gamma} c_l(\boldsymbol{k}) \mathbf{T}_l(\boldsymbol{k}) \right\rangle_F \quad \no \\
    &= C_N\sum_{\boldsymbol{k} \in  \hat{S}_k}\sum_{j,l \in \Gamma} c_j(\boldsymbol{k}) \overline{c_l(\boldsymbol{k})} \lambda_{\boldsymbol{k},\kappa_2}^{(j)}\left\langle \mathbf{T}_j(\boldsymbol{k}), \mathbf{T}_l(\boldsymbol{k}) \right\rangle_F \quad  \no\\
    &= C_N\sum_{\boldsymbol{k} \in  \hat{S}_k}\sum_{j,l \in \Gamma} c_j(\boldsymbol{k}) \overline{c_l(\boldsymbol{k})} \lambda_{\boldsymbol{k},\kappa_2}^{(j)}\delta_{jl} = C_N\sum_{\boldsymbol{k} \in  \hat{S}_k}\sum_{j \in \Gamma} \lambda_{\boldsymbol{k},\kappa_2}^{(j)}|c_j(\boldsymbol{k})|^2.\label{eq_proof_discrete_sum}
\end{align}

By completing the square and using the condition that $L_1 > 0,\kappa_2>\frac{L_4^2}{2L_1}$, we have
\begin{equation*}
   \max_{j\in \Gamma} \lambda_{\boldsymbol{k},\kappa_2}^{(j)}=-L_1 |\boldsymbol{k}|^2 + |L_4||\boldsymbol{k}|-\kappa_2 = -L_1\left(|\boldsymbol{k}| - \frac{|L_4|}{2L_1}\right)^2 + \frac{L_4^2}{4L_1}-\kappa_2\leq0,
\end{equation*}
which implies that $\lambda_{\boldsymbol{k},\kappa_2}^{(j)} \leq 0$ for all discrete wavevectors $\boldsymbol{k} \in  \hat{S}_k$ and for all $j \in \Gamma$.
Since every term in the sum \eqref{eq_proof_discrete_sum} is a product of a non-positive eigenvalue $\lambda_{\boldsymbol{k},\kappa_2}^{(j)}$ and a non-negative squared modulus $|c_j(\boldsymbol{k})|^2$, the entire sum must be non-positive.
Thus, we conclude that
\begin{align*}
   \il{\mathcal{L}_{h,\kappa_2}Q_h}{Q_h} &\leq 0,
\end{align*}
which establishes the dissipativity of the discrete operator $\mathcal{L}_{h,\kappa_2}$.
Then for any $\mu >0$, we have
\begin{align*}
    \|(\mu I - \mathcal{L}_{h,\kappa_2})Q_h\|_2^2 &= \il{(\mu I - \mathcal{L}_{h,\kappa_2})Q_h}{(\mu I - \mathcal{L}_{h,\kappa_2})Q_h} \\
    &= \mu^2 \|Q_h\|_2^2 - 2\mu \il{\mathcal{L}_{h,\kappa_2}Q_h}{Q_h} + \|\mathcal{L}_{h,\kappa_2}Q_h\|_2^2 \geq \mu^2 \|Q_h\|_2^2.
\end{align*}
By the Lumer-Phillips theorem and the process  stated in Lemma \ref{3.1} and \cite{du2021}, we can drive that $\mathcal{L}_{h,\kappa_2}$ generates a $C_0$ semigroup of contractions on the Hilbert space $(\mathcal{M}_h)$. Therefore, we have
\begin{align*}
    \|e^{\tau\mathcal{L}_{h,\kappa_2} }Q_h\|_2 \leq \|Q_h\|_2, \quad \forall \tau \geq 0.
\end{align*}
This completes the proof and establishes the dissipativity of the discrete operator $\mathcal{L}_{h,\kappa_2}$.
\end{proof}

\subsubsection{ Efficient Implementation of the ETD Schemes}
Let  $\varphi$ functions be defined in \eqref{var_functions}.
Using the above notations, we can express the actions of the operator exponentials $\varphi_\gamma(\mathcal{L}_h \tau)$, $\gamma = 0, 1, 2$ as follows:
\begin{align*}
	\varphi_\gamma(\mathcal{L}_h \tau) \widehat{Q}_{\boldsymbol{k}} &=  \left( \varphi_\gamma(\widehat{\mathcal{L}}_{\boldsymbol{k}} \tau) \widehat{Q}_{\boldsymbol{k}} \right) = \left( \sum_{j\in \Gamma}\varphi_\gamma(\lambda_{\boldsymbol{k}}^{(j)} \tau) (\widehat{Q}_{\boldsymbol{k}}^{}, \tensor{T}_j(\vu{k}))_F\tensor{T}_j(\vu{k})\right), \quad \gamma = 0, 1, 2.
\end{align*}
Then, for the ETD1 scheme \eqref{etd1h}  implemented with the HSTF basis  in Fourier space,   we have
\begin{align*}
	\widehat{Q}^{m+1}_{\boldsymbol{k}} = \sum_{j\in \Gamma}\lambda_{\boldsymbol{k}}^{(j)} (\widehat{Q}_{\boldsymbol{k}}^{}, \tensor{T}_j(\vu{k}))_F\tensor{T}_j(\vu{k}) + \tau \sum_{j\in \Gamma}\varphi_1(\lambda_{\boldsymbol{k}}^{(j)} \tau) (\widehat{\mathcal{N}_h(Q_{m})}_{\boldsymbol{k}}^{}, \tensor{T}_j(\vu{k}))_F\tensor{T}_j(\vu{k}).
\end{align*}
And for the ETDRK2 scheme \eqref{etd1h}  implemented with the HSTF basis  in Fourier space,   we have
\begin{align*}
	\widehat{\widetilde{Q}}^{m+1}_{\boldsymbol{k}} &= \sum_{j\in \Gamma}\lambda_{\boldsymbol{k}}^{(j)} (\widehat{Q}_{\boldsymbol{k}}^{}, \tensor{T}_j(\vu{k}))_F\tensor{T}_j(\vu{k}) + \tau \sum_{j\in \Gamma}\varphi_1(\lambda_{\boldsymbol{k}}^{(j)} \tau) (\widehat{\mathcal{N}_h(Q_{m})}_{\boldsymbol{k}}^{}, \tensor{T}_j(\vu{k}))_F\tensor{T}_j(\vu{k}) \\
	\widehat{Q}^{m+1}_{\boldsymbol{k}} &= \widehat{\widetilde{Q}}^{m+1}_{\boldsymbol{k}} + \tau \sum_{j\in \Gamma}\varphi_2(\lambda_{\boldsymbol{k}}^{(j)} \tau) \left( (\widehat{\mathcal{N}_h(\widetilde{Q}^{m+1})}_{\boldsymbol{k}}^{}, \tensor{T}_j(\vu{k}))_F - (\widehat{\mathcal{N}_h(Q_{m})}_{\boldsymbol{k}}^{}, \tensor{T}_j(\vu{k}))_F \right)\tensor{T}_j(\vu{k}).
\end{align*}
\begin{remark}[Physical Interpretation of the HSTF Basis]
\label{rem:physical_interpretation}
The HSTF basis provides a direct physical interpretation of the elementary structural motifs of an order parameter field in Fourier space. Each mode corresponds to a distinct pattern of spatial modulation along the wavevector $\vu{k}$:
\begin{itemize}
    \item The \textbf{$\lambda=0$} mode represents a \textbf{compressional mode}, where the principal axis is transverse to $\vu{k}$ while the degree of order is modulated longitudinally.

    \item The \textbf{$\lambda=\pm 1$} modes describe a \textbf{conical-helical} structure, where the principal axis is tilted with respect to $\vu{k}$ and precesses around it.

    \item The \textbf{$\lambda=\pm 2$} modes correspond to the quintessential \textbf{double-twist} structure. Here, the principal axis lies in the transverse plane and rotates by $4\pi$ over one spatial period. This mode is the fundamental building block of chiral liquid crystal phases, such as cholesterics and blue phases.
\end{itemize}
Thus, complex structures can be understood as a superposition of these few fundamental modes. For example, the cubic blue phases emerge from the interference of a discrete set of double-twist modes oriented along specific crystallographic axes.
\end{remark}
\section{Fully discrete  energy dissipation and error analysis}\label{section5}~\\
In this section, we present the error analysis and energy dissipation for the fully discrete ETD1 scheme \eqref{etd1h} and ETDRK2 scheme \eqref{etd2h}. The main results are presented in the following theorems.
\begin{theorem}\label{energy_d}
    Define the discrete form of the free energy \eqref{free_energy}  as follows:
\begin{align}
	E_h[Q_h]=  \frac{L_1}{2}\|\nabla_h Q_h\|_2^2+\frac{L_4}{2}\il{Q_h}{\nabla_h\times Q_h}+\frac{\alpha}{2} \| Q_h\|_2^2-\frac{\beta}{3}  \il{Q_h}{Q_h^2}+\frac{\gamma}{4} \|Q_h\|_2^4.\label{eq:fully_discrete_energy_physical}
\end{align}
Suppose   $\{Q_h^m\}_{m \geq 0}$ generated by the fully discrete ETD1 scheme \eqref{etd1h} and  ETDRK2 scheme \eqref{etd2h}. Then when $\kappa_2 \geq \frac{L_4^2}{2 L_1}$, $m \geq 0$, we have the following energy dissipation law:
       \begin{align*}
      E_h[Q_h^{m+1}] \leq E_h[Q_h^m].
       \end{align*}
\end{theorem}
\begin{proof}
    It is easy to find that the most process of Theorem \ref{energy_stability} are Lemma \ref{3.4} and Lemma \ref{3.5}. If the fully discrete level of Lemma \ref{3.4} and Lemma \ref{3.5} hold, then the energy dissipation of Theorem \ref{energy_d} can be proved as the semi-discrete case.
    Using \eqref{etd1_bound} and \eqref{etd2_bound}, we can get the uniform $L_\infty$ bound of the numerical solution $Q_h^m$ generated by the ETD1 scheme \eqref{etd1h} and the ETDRK2 scheme \eqref{etd2h}.  Then the discrete analogue of Lemma~\ref{3.4} holds:
	 \begin{align*}
		 \il{F_b(Q_h^{m+1}) - F_b(Q_h^m)}{ I} +  \il{Q_h^{m+1} - Q_h^m}{ f(Q_h^m)}  \leq \kappa_1\|Q_h^{m+1} - Q_h^m\|_2^2,
	 \end{align*}
Based on Lemma \ref{5_1}, when $\kappa_2 \geq \frac{L_4^2}{4 \tilde{L}_1 }= \frac{L_4^2}{2 L_1 }$, we can  show that the discrete analogue of Lemma~\ref{3.5} holds:
	 \begin{align*}
		 &\il{F_e(Q_h^{m+1}) - F_e(Q_h^m)}{I} +  \il{Q_h^{m+1} - Q_h^m}{\mL_0(Q_h^{m+1})} \\ &\leq - \frac{L_1}{2} \|\nabla_h(Q_h^{m+1}-Q_h^m)\|_{2}^2- \frac{L_4}{2}\il{Q_h^{m+1}-Q_h^m}{\nabla_h\times{(Q_h^{m+1}-Q_h^m)}} \\ &\leq  \tilde{L}_1 \il{Q_h^{m+1}-Q_h^m}{\Delta_h(Q_h^{m+1}-Q_h^m)} - \frac{L_4}{2}\il{Q_h^{m+1}-Q_h^m}{\nabla_h\times(Q_h^{m+1}-Q_h^m)}\\&\leq \il{Q_h^{m+1}-Q_h^m}{\mL_{h,\kappa_2}(Q_h^{m+1}-Q_h^m)}+\kappa_2 \|Q_h^{m+1}-Q_h^m\|_{2}^2 \leq \kappa_2 \|Q_h^{m+1}-Q_h^m\|_{2}^2,
	 \end{align*}
where $\tilde{L}_1 = \frac{L_1}{2}$.
Thus, the energy dissipation of Theorem \ref{energy_d} holds.
\end{proof}
\begin{theorem}\label{etd_error1}
	Suppose the exact solution \(Q(t)\) of the  BPs equation \eqref{1.6} is sufficiently smooth,  \(Q \in C^4([0,T]; H^s_{\text{per}}(\Omega))\) with \(s > 3/2\) and $m \geq 0$. Let \(Q_h^m\) be the numerical solution of the ETD1 scheme \eqref{etd1}.  Then when $\tau \le (2C_{inv}C_{err,1})^{-\frac{1}{3-r}} h^{\frac{3}{2(3-r)}}, r \in [0,3)$, there exists a positive constant \(C\), independent of \(N\) and \(\tau\), such that
	\begin{align}
		\|Q(t_m) - Q_h^m\|_{2} \leq C (\tau + h^s).
	\end{align}
And when $\tau \le (2C_{inv}C_{err,1})^{-\frac{1}{3-r}} h^{\frac{3}{2(3-r)}}, r \in [1,3)$, there exists a positive constant \(C\), independent of \(N\) and \(\tau\), such that
	\begin{align}
		\|Q(t_m) - Q_h^m\|_{\infty} \leq C (\tau + h^{s-\frac{3}{2}}).
	\end{align}
\end{theorem}

\begin{theorem}\label{etd_error2}
	Suppose the exact solution \(Q(t)\) of the  BPs equation \eqref{1.6} is sufficiently smooth, \(Q \in C^4([0,T]; H^s_{\text{per}}(\Omega))\) with \(s > 3/2\) and $m \geq 0$. Let \(Q_h^m\) be the numerical solution of the ETDRK2 scheme \eqref{etd2}.  Then when $\tau \le (2C_{inv}C_{err,2})^{-\frac{1}{3-r}} h^{\frac{3}{2(3-r)}}, r \in [0,3)$, there exists a positive constant \(C\), independent of \(N\) and \(\tau\), such that

	\begin{align}
		\|Q(t_m) - Q_h^m\|_{2} \leq C (\tau^2 + h^s).
	\end{align}
And when $\tau \le (2C_{inv}C_{err,2})^{-\frac{1}{3-r}} h^{\frac{3}{2(3-r)}}, r \in [2,3)$, there exists a positive constant \(C\), independent of \(N\) and \(\tau\), such that
	\begin{align}
		\|Q(t_m) - Q_h^m\|_{\infty} \leq C (\tau^2 + h^{s-\frac{3}{2}}).
	\end{align}
\end{theorem}
The proofs of Theorems \ref{etd_error1}-\ref{etd_error2} will be given in the subsequent subsections.
\subsection{Convergence Analysis for the ETD1 Scheme}~

\textbf{High-Order Consistency Analysis via Constructed Solutions.}
A standard error analysis for the first-order ETD1 scheme yields a local truncation error of \(\mathcal{O}(\tau^2)\). This low-order consistency is insufficient for establishing a uniform \(L^\infty\) bound without imposing an overly restrictive time step constraint. To remedy this, we introduce a higher-order approximate solution denoted by \(\check{Q}(t)\).
Let \(Q(t)\) be the exact solution to \eqref{1.8}. The constructed solution \(\check{Q}(t)\) for the ETD1 scheme associated with time step \(\tau\) is defined as
\begin{equation}
    \check{Q}(t) = Q(t) +  \tau  Q_{\tau,1}(t) +  \tau ^2 Q_{\tau,2}(t),
\end{equation}
where the correction terms \(Q_{\tau,1}(t)\) and \(Q_{\tau,2}(t)\) depend  on \(Q(t)\) and its time derivatives, which  are constructed to ensure that \(\check{Q}(t)\) satisfies the numerical scheme up to a high-order residual. We now  provide the specific construction process of auxiliary functions \(Q_{\tau,1}\) and \(Q_{\tau,2}\).

\textbf{Step 1:} Construct the correction function \(Q_{\tau,1}\).
Let \(Q_N(t) = \PN Q(t)\) be the standard Fourier projection of the exact solution \(Q(t)\). Taking the Fourier projection of   \eqref{1.8} over the interval \([t_m, t_{m+1}]\), we derive the evolution equation satisfied by \(Q_N(t)\):
\begin{equation}
	\frac{d Q_N}{dt} = \Lh Q_N + \Nh(Q_N(t)) + \mathcal{O}\left(h^{s}\right),\quad t \in [t_m, t_{m+1}].\label{f6}
\end{equation}
To analyze the temporal error, we expand the nonlinear term \(\Nh(Q_N(t))\) around \(t=t_m\):
\begin{equation}
    \Nh(Q_N(t)) = \Nh(Q_N(t_m)) +  \tau G_1(t_m) + \mathcal{O}( \tau ^2),\label{fiv5}
\end{equation}
 where \(G_1(t_m) := \frac{d}{dt}\Nh(Q_N(t))\Big|_{t=t_m} = \Nh'(Q_N(t_m))Q_{N,t}(t_m)\).
Substituting \eqref{fiv5} into \eqref{f6}, we obtain
\begin{equation}
    \frac{d Q_N}{dt} = \Lh Q_N + \Nh(Q_N(t_m)) +  \tau  G_1(t_m) + \mathcal{O}(\dt^2 + h^s),\quad t \in [t_m, t_{m+1}]. \label{eq:q_n_expanded}
\end{equation}

The first-order temporal correction function \(Q_{\tau,1}(t)\) is designed to cancel the leading-order error term \( \tau  G_1(t_m)\) in the above equation.
Therefore, \(Q_{\tau,1}(t)\) is defined as the solution to the following linear ordinary differential equation (ODE) system for its spectral coefficients:
\begin{equation} \label{eq:q1_ode}
	\begin{cases}
		\frac{d Q_{\tau,1}(t)}{dt} = \Lh Q_{\tau,1}(t) + \Nh'(Q_N(t_m))Q_{\tau,1}(t_m) - G_1(t_m), \quad t \in [t_m, t_{m+1}], \\
		Q_{\tau,1}(0) = \mathbf{0}.
	\end{cases}
\end{equation}
This is a linear system with constant coefficients over the interval \([t_m, t_{m+1}]\).
The existence and uniqueness of its solution are standard. The solution \(Q_{\tau,1}(t)\) depends only on the state \(Q_N(t_m)\).

\textbf{Step 2:} Construct the correction function $Q_{\tau,2}$.
Now, let us define a first-order corrected solution \(\check{Q}_{\tau,1}(t) = Q_N(t) + \dt P_N Q_{\tau,1}(t)\). We now derive the evolution equation satisfied by \(\check{Q}_{\tau,1}(t)\). Differentiating \(\check{Q}_{\tau,1}(t)\) with respect to time and substituting the equations for \(Q_N(t)\) and \(Q_{\tau,1}(t)\), we obtain:
\begin{equation}
    \frac{d \check{Q}_{\tau,1}}{dt} = \Lh \check{Q}_{\tau,1} + \Nh(\check{Q}_{\tau,1}(t_m)) + \dt^2 G_2(t_m) + \mathcal{O}(\dt^3 + h^s), \quad t \in [t_m, t_{m+1}], \label{eq:u1_hat_evolution}
\end{equation}
where \(G_2(t_m)\) represents the next-order error term  that depends on \(Q_N(t_m)\) and its derivatives.

Following the same logic, the second-order temporal correction function, \(Q_{\tau,2}(t)\), is designed to cancel the next-order error term \(\dt^2 G_2(t_m)\) from the evolution of \(\check{Q}_{\tau,1}(t)\).

Thus, \(Q_{\tau,2}(t)\) is defined as the solution to the following linear ODE system:
\begin{equation} \label{eq:q2_ode}
    \begin{cases}
        \frac{d Q_{\tau,2}(t)}{dt} = \Lh Q_{\tau,2}(t) + \Nh'(Q_N(t_m))Q_{\tau,2}(t_m) - G_2(t_m), \quad t \in [t_m, t_{m+1}], \\
        Q_{\tau,2}(0) = \mathbf{0}.
    \end{cases}
\end{equation}
This system is also linear with constant coefficients over the interval \([t_m, t_{m+1}]\). The existence and uniqueness of its solution are guaranteed under standard assumptions. The solution \(Q_{\tau,2}(t)\) depends only on the state \(Q_N(t_m)\) and the previously constructed correction \(Q_{\tau,1}(t)\).

\textbf{Step 3:}  Construct the final constructed solution \(\check{Q}(t)\).
Define the grid function \(\check{Q}(t) := \check{Q}_{\tau,1}(t) + \dt^2 P_N Q_{\tau,2}(t)\). Combining \eqref{eq:u1_hat_evolution} and \eqref{eq:q2_ode} leads to the evolution equation for \(\check{Q}(t)\):
\begin{equation}
    \frac{d \check{Q}}{dt} = \Lh \check{Q} + \Nh(\check{Q}(t_m)) + \mathcal{O}(\dt^3 + h^s), \quad t \in [t_m, t_{m+1}]. \label{u_final_evolution}
\end{equation}
To arrive at this result, we have used the fact that
\begin{align*}
    \Nh(\check{Q}(t)) &= \Nh(\check{Q}_{\tau,1}(t) + \dt^2 Q_{\tau,2}(t))
	= \Nh(\check{Q}_{\tau,1}(t)) + \dt^2 \Nh'(\check{Q}_{\tau,1}(t)) Q_{\tau,2}(t) + \mathcal{O}(\dt^4) \\
	&= \Nh(\check{Q}_{\tau,1}(t_m)) + \dt^2 \Nh'(\check{Q}_{\tau,1}(t_m)) Q_{\tau,2}(t) + \mathcal{O}(\dt^3) \\
	&= \Nh(\check{Q}(t_m)) + \mathcal{O}(\dt^3).
\end{align*}
The evolution equation \eqref{u_final_evolution} demonstrates that our constructed solution \(\check{Q}(t)\) satisfies the semi-discrete PDE up to a residual of order \(\mathcal{O}(\dt^3 + h^s)\). This is precisely the high-order consistency required to proceed with the main error analysis.

\begin{remark}
To facilitate a high-order error analysis, we employ a constructive approach to define a modified solution $\check{Q}(t)$. This method is systematic and hierarchical, ensuring the existence, uniqueness, and smoothness of the required correction functions, $Q_{\tau,1}(t)$ and $Q_{\tau,2}(t)$, provided the exact solution $Q(t)$ is sufficiently regular and standard assumptions on the operator $\mathcal{L}_h + \mathcal{N}_h'(Q)$ hold. The key outcome of this construction is that the local truncation error of the ETD1 scheme with respect to $\check{Q}(t)$ is elevated to order $\mathcal{O}(\Delta t^3)$, which is the cornerstone of our subsequent analysis.
\end{remark}

\textbf{\texorpdfstring{$L^\infty$}{L-infinity} Boundedness and Error Analysis.}
We define the  discrete error as \({e}^m :=  Q_h^m-{Q}_N(t_m) \) and
the modified discrete error as \(\check{e}^m :=  Q_h^m-\check{Q}_N(t_m) \).
Then we can obtain the relationship between the two errors:
\begin{align}
    {e}^m = \check{e}^m + \tau P_N Q_{\tau,1}(t_m) + \tau^2 P_N Q_{\tau,2}(t_m). \label{f21}
\end{align}
Integrating the semi-discrete equation \eqref{u_final_evolution} from \(t_m\) to \(t_m + \tau\), we obtain:
\begin{align}
	\check{Q}_N(t_m + \tau) = e^{\tau \Lh} \check{Q}_N(t_m) + \varphi_1( \tau \Lh)[ \Nh(\check{Q}_N(t_m))+ \check{\mathbf{R}}_{h\tau}^{(1)} ], \label{f8}
\end{align}
where \(\check{\mathbf{R}}_{h\tau}^{(1)}\) is the local truncation error term. And  through \eqref{u_final_evolution}, we also have
\begin{equation}
	\|\check{\mathbf{R}}_{h\tau}^{(1)}\|_{2} \le C_{R_1}(\tau^3 + h^s), \label{for3}
\end{equation}
with \(C_{R_1}\) being a constant depending on the exact solution \(Q\) and the final time \(T\).

The main steps of our analysis are as follows:
 We first assume a uniform \(L^\infty\) bound for the modified discrete error \(\check{e}^m\) up to time \(t_m\). This assumption allows us to establish the \(L^\infty\) boundedness of the numerical solution \(Q_h^m\) itself. With this boundedness, we can then derive a high-order estimate for the \(L^2\) norm of the error \(\check{e}^m\). Finally, we verify that the initial bootstrap assumption holds true under a suitable CFL condition on the time step \( \tau \) and spatial mesh size \( h \).

\textbf{Step 1:} We make a bootstrap assumption for the modified discrete error \(\check{e}^m\).
Assume there exists a time \(T > 0\) such that for all \(m \tau  \le T\), the modified discrete error satisfies
\begin{equation}
    \max_{0 \le m \le n} \|\check{e}^m\|_{\infty} \le 1.\label{for35}
\end{equation}
A direct consequence of this hypothesis is the \(L^\infty\) boundedness of the numerical solution \(Q_h^m\). By the triangle inequality and the definition of \(\check{e}^m\), we have
\begin{align}
	\|Q_h^m\|_{\infty} = \|\check{Q}_N(t_m) - \check{e}^m\|_{\infty} \le \|\check{Q}_N(t_m)\|_{\infty} + \|\check{e}^m\|_{\infty}.\label{f3}
\end{align}
Since \(Q\) is smooth, \(\check{Q}_N(t_m)\) is uniformly bounded, i.e., there exists a constant \(M_Q\) such that
\begin{equation}
	\|\check{Q}_N(t_m)\|_{\infty} \le M_Q, \quad \text{for all } t_m \le T.\label{f2}
\end{equation}
Then subtracting \eqref{f2} into \eqref{f3} and  using the bootstrap assumption \eqref{for35}, we can get
\begin{align}
	\|Q_h^m\|_{\infty} &\le \|\check{Q}_N(t_m)\|_{\infty} + \|\check{e}^m\|_{\infty} \le M_Q + 1. \label{for36}
\end{align}
Finally, for \(\Nh\) defined in \eqref{nonlinear_operator}, based on \eqref{f2} and \eqref{for36}, we can establish the following Lipschitz condition:
\begin{equation}
	\|\Nh(\check{Q}_N(t_m)) - \Nh(Q_h^m)\|_{2} \le K \|\check{e}^m\|_{2}, \label{eq:Nh_lipschitz}
\end{equation}
where the constant \(K\) depends on \(M_Q\) and the coefficients of \(\Nh\).

\textbf{Step 2:} We now establish a high-order estimate for the \(L^2\) norm of the modified discrete error \(\check{e}^m\).
\begin{theorem}[$L^2$ Error Estimate] \label{thm:h1_error}
Suppose the exact solution \(Q(t)\) of the  BPs equation \eqref{1.6} is sufficiently smooth,  \(Q \in C^4([0,T]; H^s_{\text{per}}(\Omega))\) with \(s > 3/2\). Let \(Q_h^m\) be the numerical solution of the ETD1 scheme \eqref{etd1}.   The modified discrete error of the ETD1 scheme \eqref{etd1h} satisfies the following estimate:
\begin{equation}
    \max_{0 \le m \le n} \|\check{e}^m\|_{2} \le C_{err,1}( \tau ^3 + h^s),
\end{equation}
for \(m \tau  \le T\), provided \(\check{e}^0=0\).
\end{theorem}
\begin{proof}
Subtracting \eqref{etd1h} from \eqref{f8}, we obtain the error equation for \(\check{e}^m\):
\begin{equation}
	\check{e}^{m+1} = e^{ \tau \Lh} \check{e}^m +  \tau  \varphi_1( \tau \Lh) \left[ \Nh(\check{Q}_N(t_m)) - \Nh(Q_h^m)+\check{\mathbf{R}}_{h\tau}^{(1)} \right]. \label{eq:error_eqn}
\end{equation}
Taking the \(L^2\) norm defined by \eqref{for1} on both sides of \eqref{eq:error_eqn} and using Lemma \ref{5_1}, \eqref{eq:Nh_lipschitz} and the triangle inequality, we can derive the following estimate:
\begin{align*}
	\|\check{e}^{m+1}\|_{2} &\le \|e^{ \tau \Lh} \check{e}^m\|_{2} +  \tau  \|\varphi_1( \tau \Lh) \left[ \Nh(\check{Q}_N(t_m)) - \Nh(Q_h^m) \right]\|_{2} + \|\tau \varphi_1( \tau \Lh)\check{\mathbf{R}}_{h\tau}^{(1)}\|_{2}\\
	&\le e^{-{\kappa_1} \tau} \|\check{e}^m\|_{2} +  \kappa_1^{-1}(1 - e^{-{\kappa_1} \tau}) \|\Nh(\check{Q}_N(t_m)) - \Nh(Q_h^m)\|_{2} + C_{R_1}\tau( \tau ^3 + h^s)\\	&\le e^{-{\kappa_1} \tau} \|\check{e}^m\|_{2} +  \kappa_1^{-1}(1 - e^{-{\kappa_1} \tau}) K \|\check{e}^m\|_{2} + C_{R_1}\tau( \tau ^3 + h^s)\\&\le  \|\check{e}^m\|_{2} +  \kappa_1^{-1}(1 - e^{-{\kappa_1} \tau}) (K - \kappa_1) \|\check{e}^m\|_{2} + C_{R_1}\tau( \tau ^3 + h^s).
\end{align*}
Since \(x^{-1}(1-e^{-x}) \le 1\) for all \(x > 0\), we have
\begin{align*}
	\|\check{e}^{m+1}\|_{2}
	&\le \|\check{e}^m\|_{2} +  \tau (K - \kappa_1  )\|\check{e}^m\|_{2} + C_{R_1}\tau( \tau ^3 + h^s)\\
	&\le \|\check{e}^m\|_{2} +  \tau K' \|\check{e}^m\|_{2} + C_{R_1}\tau( \tau ^3 + h^s),\\&\le (1+K' \tau )\|\check{e}^m\|_{2} + C_{R_1}\tau( \tau ^3 + h^s),
\end{align*}
where \(K' = K - \kappa_1\).
Based on  \(\check{e}^0 = 0\), a standard application of the discrete Gronwall's inequality for \(m \tau  \le T\) gives
\[
\|\check{e}^m\|_{2} \le \frac{C_{R_1}\tau( \tau ^3 + h^s)}{K' \tau } (e^{K'T} - 1) := C_{err,1}( \tau ^3 + h^s).
\]
\end{proof}

\textbf{Step 3:}  We verify the bootstrap assumption and establish the uniform \(L^\infty\) bound for the numerical solution \(Q_h^m\).

From \cite{Cai2017}, we have the following inverse inequality for any \(v_N \in \VN\):
\begin{align}
	\|v_N\|_{\infty} &\le C_{inv} N^{d/2} \|v_N\|_{2} \le C_{inv} h^{-\frac{d}{2}} \|v_N\|_{2}, \quad \forall v_N \in \VN, \label{for37}
\end{align}
where \(C_{inv}\) is a constant independent of \(N\) and \(h\), \(h \sim \frac{1}{N}\), and \(d\) is the spatial dimension.
Based on the inequality \eqref{for37} when $d=3$, for our error function \(\check{e}^{m+1}\), we have
\begin{align}
\|\check{e}^{m+1}\|_{\infty} &\le C_{inv} h^{-\frac{3}{2}} \|\check{e}^{m+1}\|_{2}
    \le C_{inv} h^{-\frac{3}{2}}  C_{err,1}( \tau ^3 + h^s). \label{eq:Linf_error_bound_final}
\end{align}
To close the bootstrap argument, we must show that the right-hand side of \eqref{eq:Linf_error_bound_final} is less than or equal to 1. This requires
\[
	C_{inv} C_{err,1} h^{-\frac{3}{2}}  ( \tau ^3 + h^s) \le 1.
\]
We split this condition into spatial and temporal parts and obtain the following inequalities:
\begin{align*}
	C_{inv} C_{err,1} h^{s-\frac{3}{2}}  &\le \frac{1}{2}, \quad
	C_{inv} C_{err,1} h^{-\frac{3}{2}}  \tau ^3 \le \frac{1}{2}.
\end{align*}

The spatial part requires \(h^{s-\frac{3}{2}}  \le \frac{1}{2 C_{inv} C_{err,1}}\), which is satisfied if \(s \ge \frac{3}{2}\).
The temporal part imposes the critical constraint:
\begin{align}
	2C_{inv} C_{err,1} h^{-\frac{3}{2}}  \tau ^3 &\le 1 \no \\
	\implies 2 C_{inv} C_{err,1} h^{-\frac{3}{2}}  \tau ^3 &\le \tau^r  \le 1,\text{with}\quad r \in [0,3),\no \\
	\implies \tau^{3-r} &\le (2C_{inv}C_{err,1})^{-1} h^{\frac{3}{2}}\no\\
	\implies \tau &\le (2C_{inv}C_{err,1})^{-\frac{1}{3-r}} h^{\frac{3}{2(3-r)}},\quad r \in [0,3). \label{for34}
\end{align}
Then we find that for a sufficiently small choice of \( \tau \) and \(h\) satisfying the CFL condition \eqref{for34}, there exists  \(\norm{\check{e}^{m+1}}{}_\infty \le 1\).
Consequently, the uniform \(L^\infty\) bound for the numerical solution follows directly from:
\begin{align}
    \|Q_h^{m+1}\|_{\infty} \le \|\check{Q}_N(t_{m+1})\|_{\infty} + \|\check{e}^{m+1}\|_{\infty} \le M_Q + 1. \label{etd1_bound}
\end{align}

These results validate our initial bootstrap hypothesis \eqref{for35} and \eqref{for36}, which therefore hold for all time \(m \tau  \le T\).

\textbf{Step 4:} We now present the detailed error analysis for Theorem \ref{etd_error1}.
The total error \(e_h^m = Q(t_m) - Q_h^m\) can be decomposed into the approximation error \(\eta^m = Q(t_m) - \PN Q(t_m)\) and the discrete error \(e^m = \PN Q(t_m) - Q_h^m= Q_N(t_m) - Q_h^m \). Based on the triangle inequality, we can obtain the final error estimate of \(L_2\) and \(L_\infty\) norm.

    For \(\eta^m\), using the standard approximation theory for spectral methods and the Sobolev embedding theorem, we have
    \begin{align}
        \|\eta^m\|_{2} &= \|Q(t_m) - \PN Q(t_m)\|_{2} \le C_p h^s. \label{l5_24}\\
        \|\eta^m\|_{\infty}& = \|Q(t_m) - \PN Q(t_m)\|_{\infty}\le C' h^{s-\frac{3}{2}} \|Q(t_m)\|_{H^{\frac{3}{2}}} \le C_p h^{s-\frac{3}{2}}.\label{l5_25}
    \end{align}
    For the \(L_2\) norm  of \(e_h^m\), based on \eqref{f21}, \eqref{l5_24} and Theorem \ref{thm:h1_error}, we have
\begin{align*}
	\|e_h^m\|_{2} =& \|Q_N(t_m) - Q_h^m \|_2+\|\eta^m\|_{2}\\
	=& \|Q_N(t_m) -\check{Q}_N(t_m)\|_2+\|\check{Q}_N(t_m)- Q_h^m\|_2 +C_p h^s\\
    	=& \tau \|\PN Q_{\tau,1}(t_m)\|_2+\tau^2 \|\PN Q_{\tau,2}(t_m)\|_2+\|\check{e}^m\|_2+C_p h^s\\
	=& \tau \|\PN Q_{\tau,1}(t_m)\|_2+\tau^2 \|\PN Q_{\tau,2}(t_m)\|_2+ C_{err,1}( \tau ^3 + h^s)+C_p h^s\\
    \le& C (\tau + h^{s}).
\end{align*}
For the \(L^\infty\) norm \(e_h^m\), when \(r \in [1,3)\), based on  \eqref{f21} \eqref{eq:Linf_error_bound_final}  and \eqref{l5_25}, we have
\begin{align*}
	\|e_h^m\|_{\infty} =& \|Q_N(t_m) - Q_h^m \|_\infty+\|\eta^m\|_\infty\\
	=& \|Q_N(t_m) -\check{Q}_N(t_m)\|_\infty+\|\check{Q}_N(t_m)- Q_h^m\|_\infty+C_p h^{s-\frac{3}{2}} \\
	=& \tau \|\PN Q_{\tau,1}(t_m)\|_\infty+\tau^2 \|\PN Q_{\tau,2}(t_m)\|_\infty+\tau^r+C_{inv}  C_{err,1} h^{s-\frac{3}{2}}+C_p h^{s-\frac{3}{2}}\\
	\le& C (\tau +  h^{s-\frac{3}{2}}).
\end{align*}
Then we get the conclusion of Theorem \ref{etd_error1}.
\subsection{Convergence Analysis for the ETDRK2 Scheme}~\\
In this section, we extend the previous analysis to the second-order ETDRK2 scheme \eqref{etd2h}.
 The main goal is to establish a uniform \(L^\infty\) bound for the numerical solution \(Q_h^m\) generated by the ETDRK2 scheme, under a suitable CFL-type condition on the time step \( \tau \).
 Additionally, we aim to derive an optimal \(L^2\) and \(L^\infty\) error estimate between the exact solution and the numerical solution.

\textbf{Step 1:}
 We first perform a high-order consistency analysis via constructed solutions.
Let \(Q(t)\) be the exact solution to \eqref{1.8}. We define the constructed solution \(\breve{Q}(t)\) for the ETDRK2 scheme as
\begin{equation}
	\breve{Q}(t) = Q(t) +  \tau ^2 Q_{\tau,3}(t),\label{l5_26}
\end{equation}
where the correction terms \(Q_{\tau,3}(t)\) depends  on \(Q(t)\) and its time derivatives, which  are constructed to ensure that \(\breve{Q}(t)\) satisfies the numerical scheme up to a high-order residual.
  Taking the Fourier projection of   \eqref{1.8} over the interval \([t_m, t_{m+1}]\) and using  a linear interpolation of the nonlinear term \(\Nh(Q_N(t))\) at \(t_m\) and \(t_{m+1}\), we derive the evolution equation satisfied by \(Q_N(t)\):
\begin{equation}
	\frac{d Q_N}{dt} = \Lh Q_N + (1 - \theta) \Nh(Q_N(t_m)) + \theta \Nh(Q_N(t_{m+1})) + \tau^2 G_3(t_m) + \mathcal{O}( \tau ^3+h^s),\quad\theta= \frac{\sigma}{\tau}.\label{for10}
\end{equation}

 Specially, here we replace  \(Q_N(t_{m+1})\) with \(\widetilde{Q}_N(t_{m+1})\), defined as
 \begin{align}
	\widetilde{Q}_N(t_{m+1})=e^{\tau \Lh} Q_N(t_m) + \int_{0}^{\tau} e^{(\tau - \sigma) \Lh}   \Nh(Q_N(t_m)) d\sigma.\no
 \end{align}
 Then for \eqref{f6}, we can express \(\Nh(Q_N(t))\) as
\begin{align}
	\Nh(Q_N(t)) =\Nh(Q_N(t_m+\sigma))= (1 - \frac{\sigma}{\tau}) \Nh(Q_N(t_m)) + \frac{\sigma}{\tau} \Nh(\widetilde{Q}_N(t_{m+1})) + \widetilde{\mathbf{R}}_{h\tau}^{(2)}(\sigma),
\end{align}
where $
	\widetilde{\mathbf{R}}_{h\tau}^{(2)}(\sigma)  =\Nh({Q}_N(t_m+\sigma)) -(1 - \frac{\sigma}{\tau}) \Nh({Q}_N(t_m)) - \frac{\sigma}{\tau} \Nh(\widetilde{Q}_N(t_{m+1})), \sigma \in [0,\tau].$

Then we will prove that the new truncation error term \(\widetilde{\mathbf{R}}_{h\tau}^{(2)}(\sigma)\) is still of order \(\mathcal{O}(\tau^2)\).
Letting  $\widetilde{Q}_N(t_{m+1}) $ minus $Q_N(t_{m+1})$, we can get
\begin{align}
	\widetilde{Q}_N(t_{m+1}) - Q_N(t_{m+1}) = & \int_{0}^{\tau} e^{(\tau - \sigma) \Lh} \left[ \Nh(Q_N(t_m)) - \Nh(Q_N(t_m + \sigma)) \right] d\sigma  \no\\
	= & \int_{0}^{\tau} e^{(\tau - \sigma) \Lh} \left[ -\sigma \Nh'(Q_N(t_m))Q_{t}(t_m) +\mathcal{O}(\tau^2) \right] d\sigma  \no\\
	= & -\int_{0}^{\tau} e^{(\tau - \sigma) \Lh} \sigma \Nh'(Q_N(t_m))Q_{t}(t_m) d\sigma  +\mathcal{O}(\tau^3) \no\\
	= & -\tau^2 \varphi_2(\tau \Lh) \Nh'(Q_N(t_m))Q_{t}(t_m) +\mathcal{O}(\tau^3).\label{for16}
\end{align}
Subtracting the Taylor expansion for \(Q_N(t_{m+1})\) around \(Q_N(t_m)\) into \eqref{for16}, we obtain
\begin{align}
	\widetilde{Q}_N(t_{m+1}) = & Q_N(t_m) + \tau Q_{t}(t_m) + \frac{\tau^2}{2} Q_{tt}(t_m) -\tau^2 \varphi_2(\tau \Lh) \Nh'(Q_N(t_m))Q_{t}(t_m) +\mathcal{O}(\tau^3).\label{for31}
\end{align}
Subtracting the Taylor expansion for \(\Nh(Q_N(t_m+\sigma))\) around \(\Nh(Q_N(t_m))\) and \eqref{for31} into \(\widetilde{\mathbf{R}}_{h\tau}^{(2)}(\sigma)\), we can get
\begin{align}
	\widetilde{\mathbf{R}}_{h\tau}^{(2)}(\sigma) =& \Nh({Q}_N(t_m+\sigma)) -\Nh({Q}_N(t_m)) - \frac{\sigma}{\tau} \left[ \Nh(\widetilde{Q}_N(t_{m+1})) - \Nh({Q}_N(t_m)) \right] \no\\
	=& \Nh'({Q}_N(t_m)) \left[ {Q}_N(t_m+\sigma) - {Q}_N(t_m) - \frac{\sigma}{\tau} \left( \widetilde{Q}_N(t_{m+1}) - {Q}_N(t_m) \right) \right]  +\mathcal{O}(\sigma^3) \no\\
	=& \Nh'({Q}_N(t_m)) \left[ \sigma Q_{t}(t_m) + \frac{\sigma^2}{2} Q_{tt}(t_m) - \frac{\sigma}{\tau} \left( \tau Q_{t}(t_m) + \frac{\tau^2}{2} Q_{tt}(t_m) \right. \right.\no\\ &\left. \left.-\tau^2 \varphi_2(\tau \Lh) \Nh'(Q_N(t_m))Q_{t}(t_m) \right) \right]
	  +\mathcal{O}(\sigma^3) \no\\
	=& \Nh'({Q}_N(t_m)) \left[  \left( \frac{\sigma^2}{2} - \frac{\sigma \tau}{2} \right) Q_{tt}(t_m) + \sigma \tau \varphi_2(\tau \Lh) \Nh'(Q_N(t_m))Q_{t}(t_m) \right]  +\mathcal{O}(\sigma^3) \no\\
	:= & \widetilde{G}_3(t_m) \tau^2 + \mathcal{O}(\tau^3)\label{for32}
\end{align}

Then substituting  \eqref{for32} into \eqref{for19} and substituting \eqref{for19} into the evolution equations \eqref{f6}, we can get the two steps evolution equations for \(Q_N(t)\) defined on the interval \([t_m, t_{m+1}]\):
\begin{align}
	\begin{cases}
			\frac{d \widetilde{Q}_N}{dt} = \Lh Q_N +  \Nh(Q_N(t_m)),\no\\
	\frac{d Q_N}{dt} = \Lh Q_N + (1 - \theta) \Nh(Q_N(t_m)) + \theta \Nh(\widetilde{Q}_N(t_{m+1})) + \tau^2 \widetilde{G}_3(t_m) + \mathcal{O}( \tau ^3+h^s),\no
	\end{cases}
\end{align}
 Thus, we define \(Q_{\tau,3}(t)\) as the solution to the following linear ODE system defined on the interval \([t_m, t_{m+1}]\):
\begin{equation*}
	\begin{cases}
				\frac{d \widetilde{Q}_{\tau,3}(t)}{dt} = \Lh Q_{\tau,3}(t) + \Nh'(Q_N(t_m))Q_{\tau,3}(t_m),  \\
		\frac{d Q_{\tau,3}(t)}{dt} = \Lh Q_{\tau,3}(t) + (1 - \theta)\Nh'(Q_N(t_m))Q_{\tau,3}(t_m) + \theta\Nh'(\widetilde{Q}_N(t_{m+1}))\widetilde{Q}_{\tau,3}(t_{m+1})- \widetilde{G}_3(t_{m}), \\
		Q_{\tau,3}(0) = \mathbf{0}.
	\end{cases}
\end{equation*}
This system is linear with constant coefficients over the interval \([t_m, t_{m+1}]\). The existence and uniqueness of its solution are guaranteed under standard assumptions. The solution \(Q_{\tau,3}(t)\) depends only on the state \(Q_N(t_m)\) and the previously constructed correction terms.
\begin{remark}
    Here we present why we didn't use  \(Q_N(t_{m+1})\) in \eqref{for10}  for the auxiliary function \(Q_{\tau,3}(t)\) that is more directer.
    We now  provide the ODE system  of  \(Q_{\tau,3}\) defined on the interval \([t_m, t_{m+1}]\) with \(Q_N(t_{m+1})\):
\begin{equation*}
	\begin{cases}
		\frac{d Q_{\tau,3}(t)}{dt} = \Lh Q_{\tau,3}(t) + (1 - \theta)\Nh'(Q_N(t_m))Q_{\tau,3}(t_m) + \theta\Nh'(Q_N(t_m))Q_{\tau,3}(t_{m+1})- G_3(t_{m}),  \\
		Q_{\tau,3}(0) = \mathbf{0}.
	\end{cases}
\end{equation*}
Notice that the above system is a implicit system since \(Q_{\tau,3}(t_{m+1})\) appears on the right-hand side and we can't make sure the existence and uniqueness of its solution. So inspired by the ETDRK2 scheme, we replace \(Q_N(t_{m+1})\) with \(\widetilde{Q}_N(t_{m+1})\) in \eqref{for10} to define the auxiliary function \(Q_{\tau,3}(t)\).
\end{remark}

Then  we define the constructed solution \(\breve{Q}(t) = Q(t) +  \tau ^2 Q_{\tau,3}(t)\)  and its Fourier projection \(\breve{Q}_N(t) = \PN \breve{Q}(t)\).
And we can get the following evolution equation for \(\breve{Q}_N(t)\):
\begin{align}
	\frac{d \breve{Q}_N}{dt} = \Lh \breve{Q}_N + (1 - \theta) \Nh(\breve{Q}_N(t_m)) + \theta \Nh(\widetilde{\breve{Q}}_N(t_{m+1})) + \mathcal{O}( \tau ^3+h^s),\quad t \in [t_m, t_{m+1}].\label{for33}
\end{align}
Integrating the equation \eqref{for33} from \(t_m\) to \(t_m + \tau\), we obtain:
\begin{align}
	\breve{Q}_N(t_{m+1}) =& e^{\tau \Lh} \breve{Q}_N(t_m) + \int_{0}^{\tau} e^{(\tau - \sigma) \Lh} \left[ (1 - \frac{\sigma}{\tau}) \Nh(\breve{Q}_N(t_m)) + \frac{\sigma}{\tau} \Nh(\widetilde{\breve{Q}}_N(t_{m+1})) +\breve{\mathbf{R}}_{h\tau}^{(2)}(\sigma) \right] d\sigma, \label{for17}
\end{align}
where \(\breve{\mathbf{R}}_{h\tau}^{(2)}(\sigma)\) is the truncation error term satisfying
\begin{equation}
	\|\breve{\mathbf{R}}_{h\tau}^{(2)}(\sigma)\|_{2} \le C_{R_2}( \tau ^3 + h^s). \label{for19}
\end{equation}

\textbf{Step 2:} We now present the error analysis for the constructed solution \(\breve{Q}_N(t)\) and the numerical solution \(Q_h^m\) generated by the ETDRK2 scheme \eqref{etd2h}.
The modified discrete error  is defined as \(\breve{e}^m :=  Q_h^m-\breve{Q}_N(t_m) \) and  the  discrete error   \({e}^m =  Q_h^m-{Q}_N(t_m) \).
Then we have the relation \({e}^m = \breve{e}^m - \tau^2 Q_{\tau,3}(t_m)\).

Similar to the ETD1 case, we proceed with a bootstrap argument to establish the uniform \(L^\infty\) bound of the error \({e}^m\) and consequently the numerical solution \(Q_h^m\) which satisfy
\begin{align}
	\max_{0 \le m \le n} \|\breve{e}^m\|_{\infty} \le 1,
	\max_{0 \le m \le n} \|Q_h^m\|_{\infty} \le M_0+1.\label{l5_36}
\end{align}

\begin{theorem}\label{thm:h1_error_etd2}
	Under the bootstrap assumption, the  discrete error of the ETDRK2 scheme \eqref{etd2h} satisfies
\begin{equation}
    \max_{0 \le m \le n} \|\breve{e}^m\|_{2} \le C_{err,2}( \tau ^2 + h^s),
\end{equation}
for \(m \tau  \le T\), provided \(\breve{e}^0=0\).
\end{theorem}
\begin{proof}
Subtracting \eqref{etd2h} from \eqref{for17}, we obtain the error equation for \(\breve{e}^m\):
\begin{align}
	\breve{e}^{m+1} =& e^{ \tau \Lh} \breve{e}^m +  \int_{0}^{\tau} e^{(\tau - \sigma) \Lh} (1-\frac{\sigma}{\dt})(\Nh(Q_h^m)-\Nh(\breve{Q}_N(t_m)))\no\\\quad+&\frac{\sigma}{\dt}(\Nh(\widetilde{Q}_h^{m+1})-\Nh(\breve{Q}_N(t_{m+1})))+\widetilde{\mathbf{R}}_{h\tau}^{(2)}(\sigma)d\sigma. \label{fiv22}
\end{align}
The smoothness of the exact solution \(Q\) and the boundedness of \(\mathcal N_h\) and its derivatives ensure that \(\sup_{t\in[0,T]}\|F_N''(t)\|_2\) is bounded by a constant \(C_{R_2}\) independent of \(N\) and \(\tau\).
Thus, we have established the bound for the truncation error:
\begin{equation}
	\|\widetilde{\mathbf{R}}_{h\tau}^{(2)}(\sigma)\|_{2} \le C_{R_2}( \tau ^2 + h^s). \label{for5}
\end{equation}

Then we will estimate the term \(\Nh(\widetilde{Q}_h^{m+1})-\Nh(\widetilde{\breve{Q}}_N(t_{m+1}))\). Note that \(\widetilde{Q}_h^{m+1}\) is obtained by the ETD1 scheme.
Let \(\widetilde{e}^{m+1} = \widetilde{Q}_h^{m+1}-\breve{Q}_N(t_{m+1})\). Then we have
\begin{align}
	   \widetilde{e}^{m+1} &= e^{ \tau \Lh} \breve{e}^m +  \tau  \varphi_1( \tau \Lh) \left[ \Nh(\breve{Q}_N(t_m)) - \Nh(Q_h^m)+\widetilde{\mathbf{R}}_{h\tau}^{(1)}(\sigma) \right]\no\\
	   &\le e^{-{\kappa_1} \tau} \|\breve{e}^m\|_{2} +  \kappa_1^{-1}(1 - e^{-{\kappa_1} \tau}) K \|\breve{e}^m\|_{2} )\no\\
	   &\le \|\breve{e}^m\|_{2} +  \tau K'\|\breve{e}^m\|_{2}.\label{for13}
\end{align}
Then we can get
\begin{align}
	   \Nh(\widetilde{Q}_h^{m+1})-\Nh(\breve{Q}_N(t_{m+1}))&\leq K(\|\breve{e}^m\|_{2} +  \tau K'\|\breve{e}^m\|_{2}).\no\\ &\leq K\|\breve{e}^m\|_{2} +  \tau KK'\|\breve{e}^m\|_{2}. \label{for14}
\end{align}
So we have
\begin{align}
	&(1-\frac{\sigma}{\dt})(\Nh(Q_h^m)-\Nh(\breve{Q}_N(t_m)))+\frac{\sigma}{\dt}(\Nh(\widetilde{Q}_h^{m+1})-\Nh(\breve{Q}_N(t_{m+1})))\no\\
	&\leq (1-\frac{\sigma}{\dt})K \|\breve{e}^m\|_{2} + \frac{\sigma}{\dt}(K\|\breve{e}^m\|_{2} +  \tau KK'\|\breve{e}^m\|_{2} )\no\\
	&\leq K \|\breve{e}^m\|_{2} +  \sigma KK'\|\breve{e}^m\|_{2}.\label{for15}
\end{align}

Taking the \(L^2\) norm  on both sides of \eqref{fiv22}, subtracting \eqref{for13}-\eqref{for15} into it, and using Lemma \ref{5_1}, we can derive the following estimate:
\begin{align}
	\|\breve{e}^{m+1}\|_{2} &\le \|e^{ \tau \Lh} \breve{e}^m\|_{2} +  \int_{0}^{\tau} \|e^{(\tau - \sigma) \Lh} (1-\frac{\sigma}{\dt})(\Nh(Q_h^m)-\Nh(\breve{Q}_N(t_m)))\no\\
	&\quad+\frac{\sigma}{\dt}(\Nh(\widetilde{Q}_h^{m+1})-\Nh(\breve{Q}_N(t_{m+1})))\|_{2}d\sigma + \int_{0}^{\tau}e^{ \tau \Lh}\|\breve{\mathbf{R}}_{h\tau}^{(2)}(\sigma)\|_{2}d\sigma\no\\
	&\le e^{-{\kappa_1} \tau} \|\breve{e}^m\|_{2} +  \int_{0}^{\tau} e^{-{\kappa_1} (\tau - \sigma)} (K \|\breve{e}^m\|_{2}+ \sigma K K'\|\breve{e}^m\|_{2} )d\sigma + C_{R_2}\tau( \tau ^3 + h^s)\no\\
	&\le e^{-{\kappa_1} \tau} \|\breve{e}^m\|_{2} +  \kappa_1^{-1}(1 - e^{-{\kappa_1} \tau}) K \|\breve{e}^m\|_{2} + \kappa_1^{-2}(e^{-{\kappa_1} \tau}-1+{\kappa_1} \tau )K K'\|\breve{e}^m\|_{2}\no\\
	&\quad + C_{R_2}\tau( \tau ^3 + h^s)\no\\
		&\le  \|\breve{e}^m\|_{2} +  \kappa_1^{-1}(1 - e^{-{\kappa_1} \tau}) K' \|\breve{e}^m\|_{2} + \kappa_1^{-2}(e^{-{\kappa_1} \tau}-1+{\kappa_1} \tau )K K'\|\breve{e}^m\|_{2}\no\\
	&\quad + C_{R_2}\tau( \tau ^3 + h^s)\no\\
	&\le \|\breve{e}^m\|_{2} +  \tau K' \|\breve{e}^m\|_{2} +\frac{\tau ^3}{2} K'^2\|\breve{e}^m\|_{2} + C_{R_2}\tau( \tau ^3 + h^s).\label{forr16}
\end{align}
Using $e^0=0$ and a standard discrete Gronwall's inequality, for \eqref{forr16}, we have
\begin{align*}
	\norm{\e{m}}
	&\le \sum_{k=0}^{m-1} \left(1 + \tau K' + \frac{\tau^2}{2}K'^2\right)^k  C_{R_2}\tau( \tau ^3 + h^s) \\
	&\le \sum_{k=0}^{m-1} \left(1 + \tau \left(K' + \frac{\tau}{2}K'^2\right)\right)^k C_{R_2}\tau( \tau ^3 + h^s)  \\
	&= \frac{\left(1 + \tau(K' + \frac{\tau}{2}K'^2)\right)^m - 1}{\tau(K' + \frac{\tau}{2}K'^2)}  C_{R_2}\tau( \tau ^3 + h^s)  \\
	&\le \frac{e^{T (K' + \frac{\tau}{2}K'^2)} - 1}{K' + \frac{\tau}{2}K'^2}  C_{R_2}( \tau ^3 + h^s) \\
	&= C(T, K', \kappa_1,  C_{R_2}) (\tau^2 + h^s):= C_{err,2}( \tau ^3 + h^s),
\end{align*}
for \(m \tau  \le T\). This completes the proof.
\end{proof}

\textbf{Step 3:}
We now establish the uniform \(L^\infty\) bound for the numerical solution \(Q_h^m\) generated by the ETDRK2 scheme, under a suitable CFL-type condition on the time step \( \tau \).
Similar to the ETD1 case, we use the inverse inequality to obtain the \(L^\infty\) bound for \(\breve{e}^{m+1}\):
\begin{align}
\|\breve{e}^{m+1}\|_{\infty} &\le C_{inv} h^{-\frac{3}{2}} \|\breve{e}^{m+1}\|_{2} \no\\
	&\le C_{inv} h^{-\frac{3}{2}}  C_{err,2}( \tau ^3 + h^s).
\end{align}
To close the bootstrap argument,
splitting this into spatial and temporal parts, we require
\begin{align*}
	C_{inv} C_{err,2} h^{s-\frac{3}{2}}  &\le \frac{1}{2}, \\
	C_{inv} C_{err,2} h^{-\frac{3}{2}}  \tau ^3 &\le \frac{1}{2}.
\end{align*}
Thus, we have two constraints to satisfy.
The spatial part requires \(h^{s-\frac{3}{2}}  \le \frac{1}{2 C_{inv} C_{err,2}}\), which is satisfied if \(s \ge \frac{3}{2}\).
The temporal part imposes the critical constraint:
\begin{align}
	&2C_{inv} C_{err,2} h^{-\frac{3}{2}}  \tau ^3 \le 1 \no \\
	\implies & \tau \le (2C_{inv}C_{err,2})^{-\frac{1}{3-r}} h^{\frac{3}{2(3-r)}},\quad r \in [0,3). \label{for38}
\end{align}
Under this condition \eqref{for38}, we have
\(\|\breve{e}^{m+1}\|_{\infty} \le 1\), validating the bootstrap hypothesis for all \(m \tau  \le T\).
Consequently, the uniform \(L^\infty\) bound for the numerical solution follows:
\begin{align}
    \|Q_h^{m+1}\|_{\infty} \le \|\breve{Q}_N(t_{m+1})\|_{\infty} + \|\breve{e}^{m+1}\|_{\infty} \le M_Q + 1. \label{etd2_bound}
\end{align}
These results confirm the stability and boundedness of the numerical solution generated by the ETDRK2 scheme under the specified CFL-type condition on the time step.

\textbf{Step 4:}
We now present the detailed error analysis for Theorem \ref{etd_error2}.
    Let \(e_h^m = Q(t_m) - Q_h^m\) be the total error at time \(t_m\).
    For the \(L_2\) norm  of \(e_h^m\), using  \eqref{l5_24}, \eqref{l5_26} and Theorem \ref{thm:h1_error}, we have
\begin{align*}
	\|e_h^m\|_{2} \le& \|Q_N(t_m) - Q_h^m \|_2+\|\eta^m\|_{2}\\
	\le& \|Q_N(t_m) -\breve{Q}_N(t_m)\|_2+\|\breve{Q}_N(t_m)- Q_h^m\|_2 +C_p h^s\\
	\le& \tau^2 \|\PN Q_{\tau,3}(t_m)\|_2+ C_{err,2}( \tau ^3 + h^s)+C_p h^s\\
    \le& C (\tau^2 + h^{s}).
\end{align*}
For the \(L^\infty\) norm \(e_h^m\), when \(r \in [2,3)\), using  \eqref{f21} \eqref{eq:Linf_error_bound_final}  and \eqref{l5_25}, we can obtain
\begin{align*}
	\|e_h^m\|_{\infty} =& \|Q_N(t_m) - Q_h^m \|_\infty+\|\eta^m\|_\infty\\
	=& \|Q_N(t_m) -\breve{Q}_N(t_m)\|_\infty+\|\breve{Q}_N(t_m)- Q_h^m\|_\infty+C_p h^{s-\frac{3}{2}} \\
	=& \tau^2 \|\PN Q_{\tau,3}(t_m)\|_\infty+\tau^r+C_{inv}  C_{err,2} h^{s-\frac{3}{2}}+C_p h^{s-\frac{3}{2}}\\
	\le& C (\tau^2 +  h^{s-\frac{3}{2}}).
\end{align*}
Then we get the conclusion of Theorem \ref{etd_error2}.
 \section{Numerical experiments}\label{section6}~\\
 In this section, we present some numerical experiments to validate the MBP preservation, energy dissipition and convergence properties of the fully discrete ETD schemes  \eqref{etd1h} and \eqref{etd2h} applied to  the governing equation of BPs model \eqref{1.6} with periodic boundary condition under the Frobenius norm.
    Then we present some numerical simulations to illustrate the  behaviors of  the BP III phase and the effects of the chirality.
    Through the curl term is only valid in three dimensions (3D),  in all tests the computational domain is taken as $\Omega=(0,2 \pi)^{3}$.
 \subsection{Convergence, MBP, and energy dissipation tests}\label{dim2}~\\
This subsection tests the MBP, energy dissipation and convergence orders of the  ETD schemes.
  In our tests, we take $N=64$ Fourier modes in each direction to ensure that the spatial discretization error is negligible.
The initial condition is set to be
 \begin{align*}
    Q_0(x, y, z) = c \begin{pmatrix}
        \cos(x) - \cos(y) & \sin(x)\sin(y) & \sin(x)\sin(z) \\
        \sin(x)\sin(y) & \cos(y) - \cos(z) & \sin(y)\sin(z) \\
        \sin(x)\sin(z) & \sin(y)\sin(z) & \cos(z) - \cos(x)
    \end{pmatrix}.
\end{align*}
    Based on \cite{mottramintroduction}, the parameters are set as follows: $\alpha=-1.00$, $\beta=1$, $\gamma=2.25$, $L_1=1,L_4=\frac{1}{4}$ and $c=\frac{1}{3}$.
The stability constants $\kappa_1$ and $\kappa_2$ are respectively chosen as $\kappa_1=8,\kappa_2=0.5$.
    \vskip 0.2cm
  \textbf{Convergence tests.}
  This test verifies the convergence orders of the ETD schemes in both $\mathcal{Z}$-norm and  $2$-norm.
We set the final time to $T=1$ and calculate the numerical solution with the time step size $\tau=2^{-k}\tau_{1}$, where $k=0,1,\ldots,7$ and $\tau_{1}=2^{-4}$. For the lack of the exact solution, the numerical solution generated with $\tau=2^{-8}\tau_{1}$ at $T=1$ is regarded as the benchmark solution.
  Table \ref{v2_lri_error} presents the errors and  convergence rates between the numerical solutions and the benchmark solution in both $\mathcal{Z}$-norm and  $2$-norm.  The results indicate that the  ETD1 scheme is first-order accurate in time and the ETDRK2 scheme is second-order accurate, which are consistent with Theorems \ref{etd_error1} and \ref{etd_error2}.

\begin{table}[htbp]
 \centering
 \begin{tabular}{l *{4}{c} *{4}{c}}
  \toprule
  \multicolumn{1}{c}{} & \multicolumn{4}{c}{ETD1} & \multicolumn{4}{c}{ETDRK2} \\
  \cmidrule(lr){2-9}
  \multicolumn{1}{c}{$\tau_1=2^{-4}$}
     & \multicolumn{2}{c}{$\mathcal{Z}$-norm} & \multicolumn{2}{c}{2-norm}
     & \multicolumn{2}{c}{$\mathcal{Z}$-norm} & \multicolumn{2}{c}{2-norm} \\
\cmidrule(lr){2-9}
  & Error & Rate & Error & Rate & Error & Rate & Error & Rate \\
  \midrule
  $\tau_1$       & 2.2778E-02 & -     & {3.9409E+00} & -     & 4.8009E-03 & -     & 1.0346E+00 & - \\
  $\tau_1$/2     & 1.1424E-02 & 0.995 & {1.9858E+00} & 0.988 & 1.3353E-03 & 1.846 & 2.8815E-01 & 1.844 \\
  $\tau_1$/4     & 5.6653E-03 & 1.011 & 9.8682E-01 & 1.008 & 3.4829E-04 & 1.939 & 7.5180E-02 & 1.938 \\
  $\tau_1$/8     & 2.8153E-03 & 1.008 & 4.9086E-01 & 1.007 & 8.8697E-05 & 1.973 & 1.9147E-02 & 1.973 \\
  $\tau_1$/16    & 1.4027E-03 & 1.005 & 2.4469E-01 & 1.004 & 2.2367E-05 & 1.987 & 4.8285E-03 & 1.987 \\
  $\tau_1$/32    & 7.0011E-04 & 1.002 & 1.2215E-01 & 1.002 & 5.6154E-06 & 1.994 & 1.2122E-03 & 1.993 \\
  $\tau_1$/64    & 3.4973E-04 & 1.001 & 6.1026E-02 & 1.001 & 1.4068E-06 & 1.997 & 3.0369E-04 & 1.996 \\
  $\tau_1$/128   & 1.7478E-04 & 1.000 & 3.0501E-02 & 1.000 & 3.5206E-07 & 1.998 & 7.6001E-05 & 1.998 \\
  \bottomrule
 \end{tabular}
 \caption{Errors and convergence rates computed by the ETD schemes.}
 \label{v2_lri_error}
\end{table}
 \textbf{MBP preservation and Energy dissipation tests.}
In Theorem~\ref{mbp}, we have proved that the temporal semi-discrete ETD schemes
preserve the MBP unconditionally. However, a theoretical proof for the fully
discrete schemes is challenging due to the characteristics of spectral methods.
Therefore, we now turn to numerical experiments to verify that the MBP is preserved in practice.
In this test, we set the time step size $\tau=2^{-5}$ and the final time $T=10$.
  \begin{figure}[htbp]
		 \includegraphics[width=\textwidth]{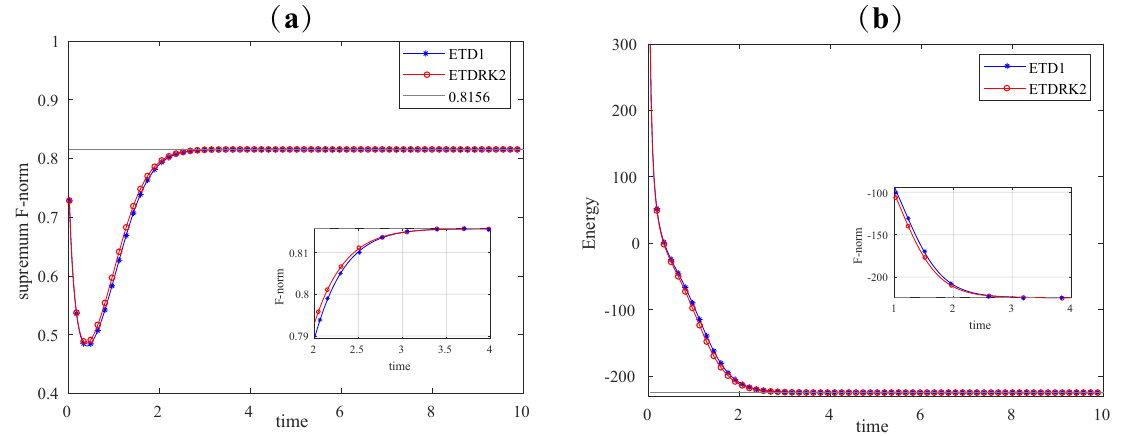}
	 \caption{Evolutions of the (a) supremum  norm $\|Q_h^m\|_\mathcal{Z}$, and (b) energies computed by the ETD1 scheme  with blue lines and ETDRK2 scheme with red lines. The dashed lines in (a) represent the bounds of the MBP, the dashed lines in (b) indicate the total energy conservation.}
		 \label{fig_mbp_norm}
 \end{figure}

 Figure~\ref{fig_mbp_norm} presents the evolution of the supremum norm $\|Q_h^m\|_\mathcal{Z}$ and the discrete free energy $E_h[Q_h^m]$. The norm in Fig. \ref{fig_mbp_norm} (a) exhibits an initial decrease, suggesting the system is evolving towards a more ordered state, followed by a subsequent increase and stabilization as the solution approaches a steady, equilibrium configuration. The overall boundedness is consistent with the preservation of the MBP (corresponds to Theorem \ref{mbp}) and demonstrates the stability of our numerical schemes.
Furthermore, as shown in Fig.~\ref{fig_mbp_norm}(b), the discrete free energy decreases monotonically, consistent with the energy dissipation laws established in Theorems~\ref{energy_stability} and \ref{energy_d}. This monotonic decay is the hallmark of a dissipative system, demonstrating that our simulation accurately captures the gradient flow dynamics inherent to the BPs model.
  \begin{figure}[htbp]
		 \includegraphics[width=\textwidth]{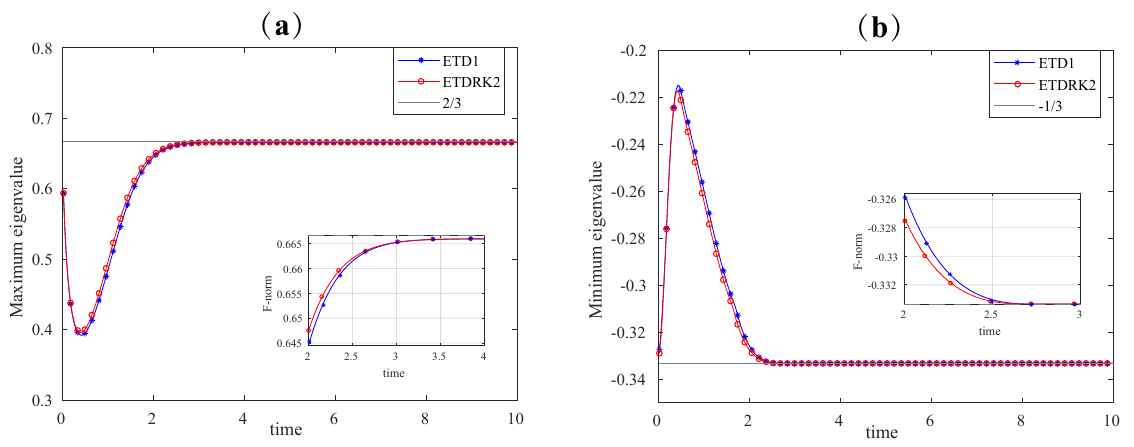}
	 \caption{Evolutions of the  (a) maximum eigenvalue of $Q_h^m$ and (b) maximum eigenvalue of $Q_h^m$  computed by the ETD1 scheme with blue lines and ETDRK2 scheme with red lines. The dashed lines in (a,b) represent the bounds of the eigenvalues, which lies in  $(-\frac{1}{3},\frac{2}{3})$.}
		 \label{fig_mbp_eign}
 \end{figure}

 Figure~\ref{fig_mbp_eign} presents the time evolution of the maximum and minimum eigenvalues of the numerical solution $Q_h^m$. The eigenvalues are shown to remain strictly within the physical bounds of $\left(-\frac{1}{3}, \frac{2}{3}\right)$ throughout the simulation, eventually converging to stable values. Adherence to these bounds is a critical indicator of a physically consistent Q-tensor evolution \cite{mottramintroduction}. These results therefore confirm that the proposed ETD schemes effectively capture the dynamics of the blue phase, yielding solutions that are both numerically stable and physically meaningful.
    \subsection{Chiral liquid crystal blue phase simulations}\label{dim3}~\\
This section presents a series of numerical simulations to illustrate the behaviors of BPs, including the blue phase III (BP III) and blue phase I (BP I).
In the free energy formulation \eqref{free_energy}, the chiral term parameter is represented by $L_4$, setting $L_4 = 2 L_1 q_0$.
 Here, $q_0$ is the chiral wave number, which is inversely proportional to the helical pitch $p_0$ and is given by $q_0 = 2\pi/p_0$~\cite{dupuis2005numerical,RevModPhys.61.385}.
To facilitate physical interpretation and parameter analysis, here we also use $q_0$ to characterize the intrinsic chirality and introduce the following dimensionless parameters as in \cite{dupuis2005numerical,grebel1984landau, RevModPhys.61.385}:
\begin{align}
\tau_c = \frac{24 \alpha \gamma}{\beta^2}, \quad
\kappa^2 = \frac{108 L_1 \gamma q_0^2}{\beta^2}.\label{dimless_params}
\end{align}
The parameter  $\tau_c$ represents the reduced temperature, indicating the proximity to the nematic-cholesteric transition and $\kappa$ quantifies the strength of chirality relative to the  elastic energy.
According to
\cite{grebel1983landau, RevModPhys.61.385,dupuis2005numerical},
the approximate stability ranges of $\kappa$ are
\begin{equation}
\text{BPI: } 0.6 \lesssim \kappa \lesssim 1.2, \qquad
\text{BPII: } 1.2 \lesssim \kappa \lesssim 1.8, \qquad
\text{BP III: } \kappa \gtrsim 1.8. \label{bp_ranges}
\end{equation}
The initial condition  is set to be a ordered state referred to \cite{dupuis2005numerical,grebel1984landau, RevModPhys.61.385}:
            \begin{align*}
 Q_0(x, y, z) = c \begin{pmatrix}
\cos(x) + \cos(y) -2\cos(z) & \sin(x)\sin(y) & \sin(x)\sin(z) \\
\sin(x)\sin(y) & \cos(y) + \cos(z) -2\cos(x) & \sin(y)\sin(z) \\
\sin(x)\sin(z) & \sin(y)\sin(z) & \cos(z) + \cos(x) -2\cos(y)
\end{pmatrix},
  \end{align*}
    where $c=0.2$.
The simulations are performed using the ETDRK2 scheme  with a time step size of $\tau = 2^{-5}$ and a spatial resolution of $64^3$ grid points. The elastic constant is set to $L_1=0.1$ to ensure the elastic constant is sufficiently small, promoting the formation of disclination networks characteristic of BPs.

 \textbf{Blue Phase III simulation.} In this simulation, we set $\kappa=3$ and $\tau_c=1$, which fall within the stability range for BP III as indicated in \eqref{bp_ranges}.
Then applying  normalization to $\beta,\gamma$ and using the relations \eqref{dimless_params}, we can determine the dimensional parameters as follows:
\begin{equation}
\alpha = 0.042, \quad \beta = 1, \quad \gamma = 1, \quad L_1 = 0.1, \quad q_0=0.9129, \quad L_4 = 0.1826.
\end{equation}
To analyze the dynamic evolution of the BP III phase, we compute and visualize several key quantities at different time points: the scalar order parameter $s$ of the Q-tensor field $Q_h^m$, the biaxiality parameter $\beta_b$, and the static structure factor $S(\boldsymbol{k})$.
Here, we give the calculation formula of $s$, $\beta_b$ \cite{mottramintroduction} and $S(\boldsymbol{k})$ \cite{henrich2011structure}:
\begin{equation}
    s=\frac{3}{2}\lambda_{\max}(Q_h^m), \quad \beta_b^2 = 1 - 6 \frac{(\text{tr} Q^3)^2}{(\text{tr} Q^2)^3},\quad S(\boldsymbol{k}) := \left \langle \widehat{Q}_{\boldsymbol{k}},\widehat{Q}_{\boldsymbol{k}} \right \rangle_F, \quad \boldsymbol{k} \in \hat{S}_h,
\end{equation}
where $\lambda_{\max}$ denotes the maximum eigenvalue, $\beta_b$ quantifies the degree of biaxiality and $S(\boldsymbol{k})$ measures the intensity of each Fourier mode.
    \begin{figure}[htbp]
            \includegraphics[width=\textwidth]{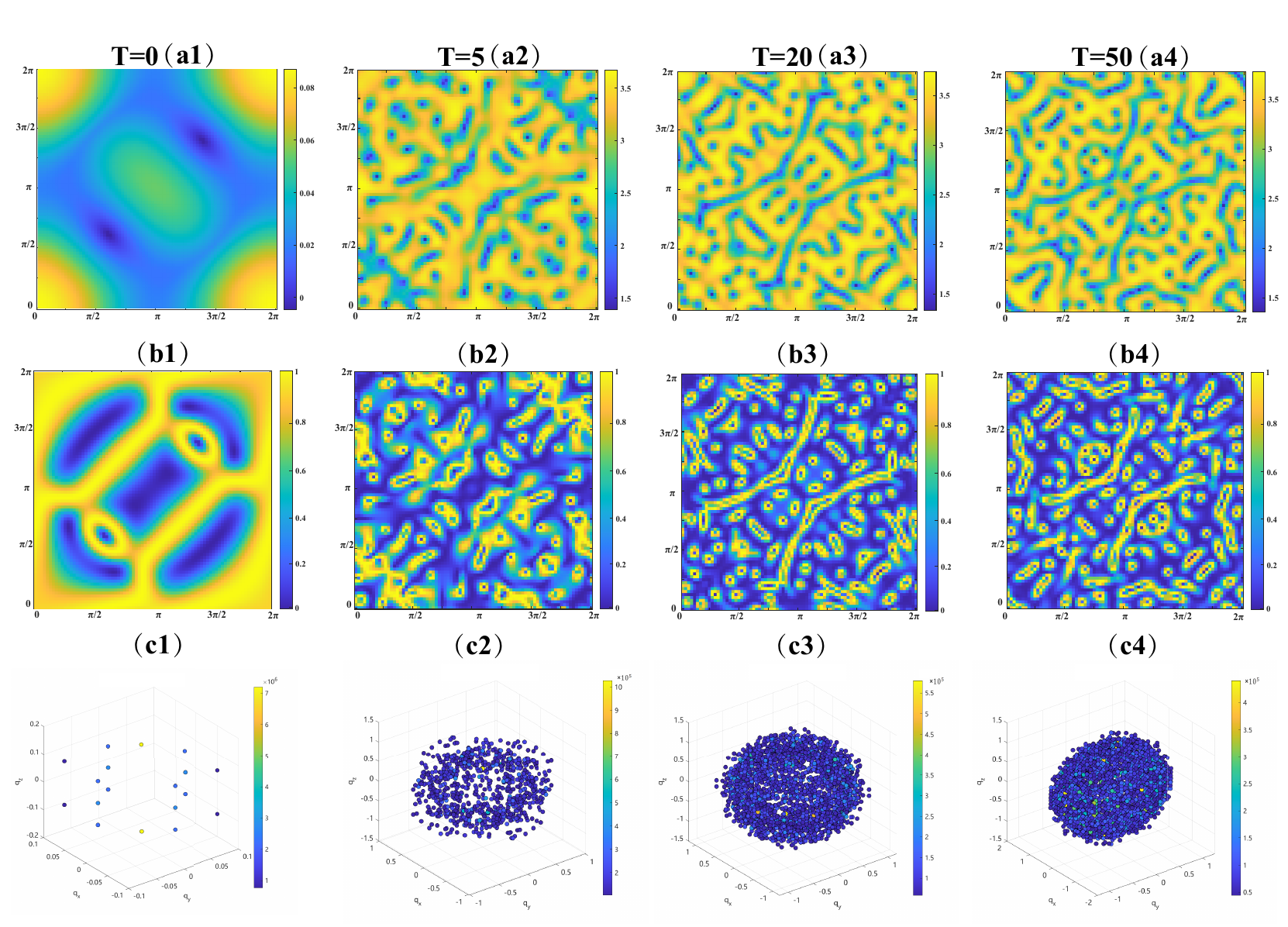}
        \caption{  The evolution of the scalar order parameter $s$ (a1-a4), the biaxiality $\beta_b$ (b1-b4), and the static structure factor $S(\boldsymbol{k})$ (c1-c4) is depicted at simulation times $T=0, 5, 20,$ and $50$. }
            \label{fig_nematic_blue3}
    \end{figure}

  Figure~\ref{fig_nematic_blue3} present the cross-sections of $s$ and  $\beta_b$ and  display the evolution of $S(\boldsymbol{k})$ depicting the dynamic evolution of the BPs equation.
  In  Fig.~\ref{fig_nematic_blue3} (a1-a4), the dynamics begin from a relatively ordered state  to a highly disordered configuration with intricate patterns and defects, maintaining the complex topology indicative of the blue phase.
  The  $\beta_b$ in (b1) starts with low values, indicating a predominantly uniaxial phase, reflecting the growing complexity of the molecular alignment.  At the final state (b4), the $\beta_b$ remains high and spatially varied, confirming the persistence of complex molecular orientations.
  The  $S(\boldsymbol{k})$ in (c1) initially displays a diffuse pattern, indicative of the lack of long-range order in the initial state. As time progresses to T=5 (c2), the scattering intensity begins to concentrate, signaling the development of short-range correlations. By T=20 (c3), a pronounced spherical shell structure emerges in reciprocal space, characteristic of BP III, indicating the presence of a characteristic length scale despite the absence of long-range translational order. Finally, at T=50 (c4), the $S(\boldsymbol{k})$ solidifies into a well-defined spherical shell, confirming the establishment of the BP III phase with its hallmark isotropic scattering pattern.
    \begin{figure}[htbp]
        \centering
            \includegraphics[width=0.8\textwidth]{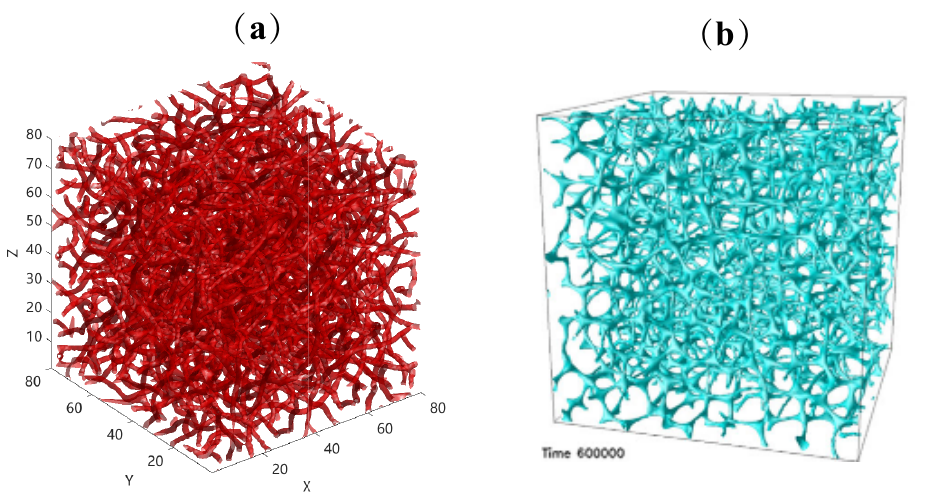}
        \caption{ (a) displays the isosurface plot of the  order parameter $Q_h^n$ with $s=0.3*s_{max}$ at $T=50$ where $s_{max}$ is the maximum value of $s$.  (b) displays disclination line network of BP III of the isosurface $s=0.12$ at $\tau_c=-0.25, \kappa=2.5$ from  \cite{henrich2011structure}.}
            \label{fig_nem-blue4}
    \end{figure}

    Figure~\ref{fig_nem-blue4}  present the isosurface with $s=0.3*s_{max}$ and the isosurface of the End-state disclination line network from \cite{henrich2011structure}.
In Fig.~\ref{fig_nem-blue4} (a), the isosurface of  $s$  reveals a highly interconnected, amorphous, sponge-like network. This morphology is topologically identical to the canonical disclination line network of BP III, shown in Fig.~\ref{fig_nem-blue4} (b). This direct structural correspondence strongly supports the conclusion that our simulation has successfully captured the complex topology of BP III.

\begin{figure}[htbp]
        \centering
            \includegraphics[width=0.6\textwidth]{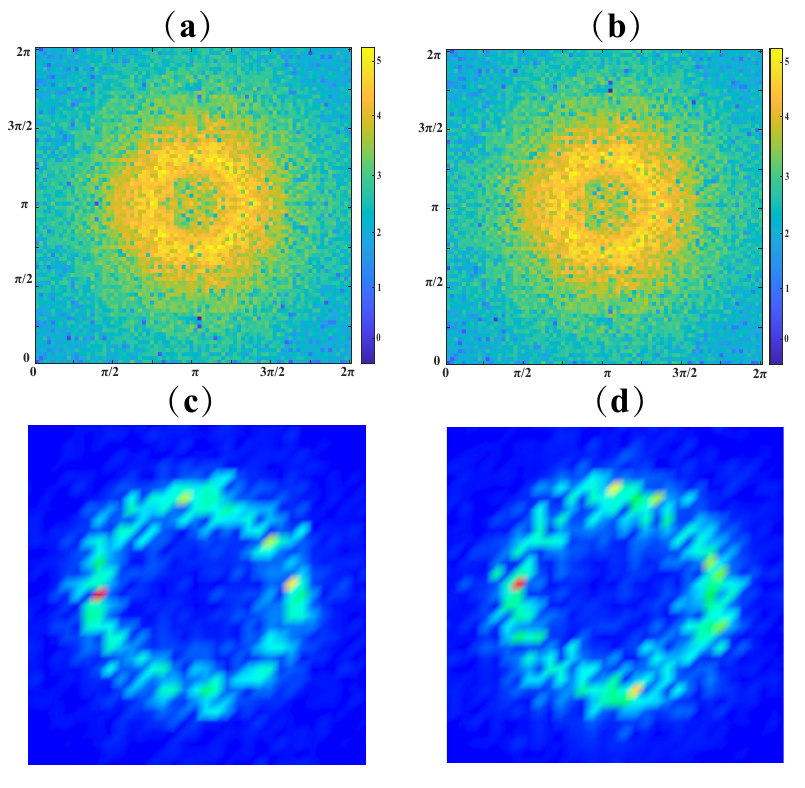}
        \caption{  (a-b) show the  static structure factor $S(\mathbf{k})$ plotted on a logarithmic scale onto the $y=0$ plane  and $x=0$ plane at $T=50$. (c-d) display the
corresponding structure factor of BP III on cuts along $y=0$ and $x=0$  from \cite{henrich2011structure}.}
            \label{fig_nem_structure_factor}
    \end{figure}
    Figure~\ref{fig_nem_structure_factor} present the $S(\mathbf{k})$  plotted on a logarithmic scale  and the reference diffraction patterns from \cite{henrich2011structure}.
 The 2D projections of $S(\mathbf{k})$ onto the $xy$-plane (Figs.~\ref{fig_nem_structure_factor}(a) and (b)) exhibit a prominent and diffuse scattering ring. This ring is the definitive hallmark of a structure with short-range order at a characteristic length scale but lacking long-range crystalline periodicity, a key feature of BP III. This scattering pattern is in excellent agreement with the reference diffraction patterns shown in Figs.~\ref{fig_nem_structure_factor}(c) and (d).
The strong correspondence in both the real-space network topology and the reciprocal-space scattering patterns confirms that our simulation has successfully captured the complex, amorphous structure of BP III.

\begin{figure}[htbp]
        \centering
            \includegraphics[width=0.8\textwidth]{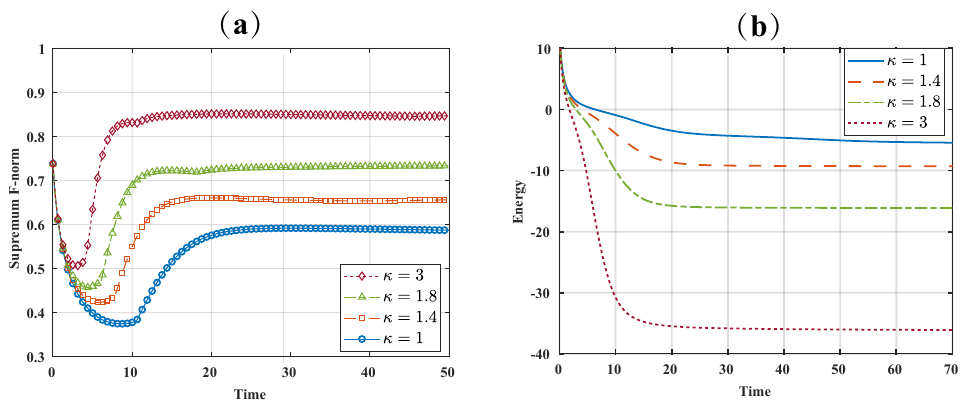}
        \caption{  (a) show the  evolution of the supremum norm $\|Q_h^m\|_\mathcal{Z}$, and (b) show the evolution of the discrete free energy $E_h[Q_h^m]$, respectively at $\kappa=1,1.4,1.8,3$.}
            \label{fig_nem_energy}
    \end{figure}
    Figure~\ref{fig_nem_energy} present the evolution of the supremum norm $\|Q_h^m\|_\mathcal{Z}$ and the discrete free energy $E_h[Q_h^m]$ at different chirality strengths $\kappa=1,1.4,1.8,3$.
From Fig.~\ref{fig_nem_energy} (a), we can observe that for all chirality strengths, the supremum norm $\|Q_h^m\|_\mathcal{Z}$ remains bounded and increases proportionally with $\kappa$, which is consistent with the MBP preservation property established in Theorem~\ref{mbp}.
Furthermore, Fig.~\ref{fig_nem_energy} (b) shows that the discrete free energy $E_h[Q_h^m]$ decreases monotonically over time for all values of $\kappa$.
Notably, higher chirality strengths lead to a more rapid decrease in free energy, indicating that increased chirality accelerates the system's relaxation towards equilibrium.

\textbf{Blue Phase I simulation.}  Next, we turn our attention to BP I. Based on the established stability region for BP I in \eqref{bp_ranges}, we set the chirality parameter to $\kappa = 1.0$. Then the other dimensional parameters are derived as follows:
\begin{equation}
\alpha = 0.042, \quad \beta = 1, \quad \gamma = 1, \quad L_1 = 0.1, \quad q_0=0.3043, \quad L_4 = 0.0609.
\end{equation}
\begin{figure}[htbp]
    \centering
            \includegraphics[width=1\textwidth]{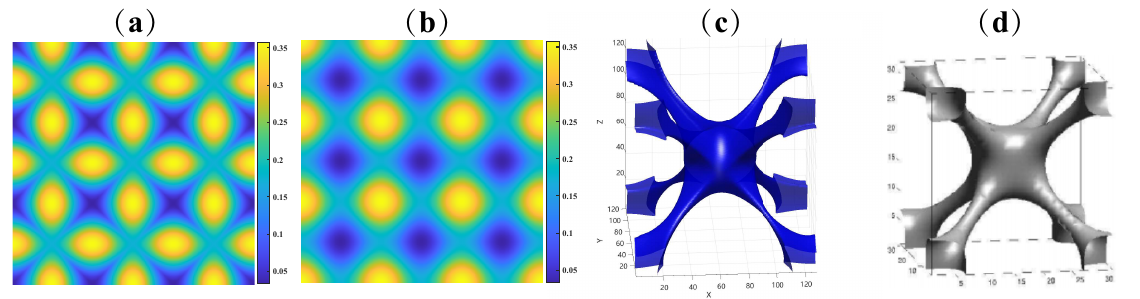}
        \caption{(a) and (b)  show the cross-sections of the scalar order parameter of layer 1 and 2. (c) displays the disclination defects and (d) shows the disclination defects of $O^5$ symmetry group listed in \cite{dupuis2005numerical}.}\label{fig_nematic_blue1_new}
\end{figure}

Figure~\ref{fig_nematic_blue1_new} presents cross-sectional views and isosurface visualizations of the disclination line structure.
    In (a) and (b), we display the cross-sections of the scalar order parameter at two different layers, revealing the arrangement of adjacent two layers, which is a manifestation of body-centered cubic (BCC) arrangement.  The isosurfaces in (c) illustrate the disclination defect network, showcasing  $O^5$ BCC symmetry, as highlighted in (d), which corresponds to the well-known symmetry group associated with BP I \cite{dupuis2005numerical}. This structural configuration is consistent with experimental observations and theoretical predictions for BP I, confirming the accuracy of our simulation in capturing the complex topology of this liquid crystal phase.
 \section{Conclusion}\label{section7}~\\
In this paper, we have developed two efficient ETD schemes for the  BPs equation. We have rigorously proved that both the ETD1 scheme and the ETDRK2 scheme preserve the maximum bound principle (MBP) unconditionally. We have also established optimal error estimates for both schemes in the discrete $L^2$ norm. Furthermore, we have demonstrated that both schemes are unconditionally energy stable with respect to a discrete energy functional. Numerical experiments have been conducted to validate our theoretical findings, including tests for MBP preservation, energy dissipition, and convergence rates. Additionally, we have presented simulations illustrating the nematic-blue phase transition process of liquid crystals using the ETDRK2 scheme. These results confirm the effectiveness and reliability of the proposed ETD schemes for simulating the complex dynamics of liquid crystals.
 \vskip 0.2cm
{\bf Acknowledgements.} G. Ji is partially supported by the National Natural Science Foundation of China (Grant No. 12471363).
\bibliographystyle{siam}
\bibliography{S0362546X14002934}

\end{document}